\documentclass{article}
\usepackage[utf8]{inputenc}
\usepackage[T1]{fontenc}
\usepackage{textcomp}
\usepackage{amsmath}
\usepackage{amsthm}
\usepackage{amsfonts}
\usepackage{amssymb}

\usepackage{graphicx}
\usepackage{xcolor}

\usepackage{enumerate}

\usepackage{anysize}
\marginsize{2cm}{2cm}{2cm}{2.5cm}

\usepackage[unicode, colorlinks=true,bookmarksdepth=3]{hyperref}
\usepackage[capitalise, nameinlink]{cleveref}

\newtheoremstyle{thm}{}{}{\itshape}{}{\color{blue}\bfseries}{}{ }{} \theoremstyle{thm}
\newtheorem{theorem}{Theorem}[section]
\newtheorem{proposition}[theorem]{Proposition}
\newtheorem{lemma}[theorem]{Lemma}
\newtheorem{corollary}[theorem]{Corollary}
\newtheorem{definition}{Definition}[section]

\newtheoremstyle{rmk}{}{}{\small}{}{\color{blue}}{}{ }{}
\theoremstyle{thm}
\newtheorem{remark}{Remark}[section]

\renewenvironment{proof}[1][Proof]{\noindent {\bfseries \color{blue} #1: }}{{\hfill \(\Box\) \smallskip}}

\numberwithin{equation}{section}
\crefname{equation}{}{} 

\AddToHook{env/lemma/begin}{\crefalias{theorem}{lemma}}
\AddToHook{env/proposition/begin}{\crefalias{theorem}{proposition}}
\AddToHook{env/corollary/begin}{\crefalias{theorem}{corollary}}
\AddToHook{env/example/begin}{\crefalias{theorem}{example}}

\newcommand{\PP}{{\mathbb{P}}}
\renewcommand{\P}{\mathsf{P}}
\newcommand{\Pp}{\mathbf{P}}

\newcommand{\E}{\mathsf{E}}
\newcommand{\EE}{\mathbb{E}}
\newcommand{\Ee}{\mathbf{E}}
\newcommand{\Fc}{\mathcal{F}}  \newcommand*{\ONE}{\mathbb{I}}

\newcommand{\disteq}{{\ \stackrel{\mathrm{d}}{=}\ }}

\newcommand{\eps}{{\varepsilon}}
\newcommand{\N}{\mathbb{N}}

\newcommand{\Z}{\mathbb{Z}}
\newcommand{\R}{\mathbb{R}}

\newcommand{\sgn}{\mathop{\mathrm{sgn}}}

\newcommand{\uu}{\mathbf{u}}

\allowdisplaybreaks[1]
\usepackage{fancyhdr}
\begin{document}
\title{Polynomial mixing for stochastic viscous conservation law equation on the whole line}

\author{Peng Gao \thanks{School of Mathematics and Statistics, Center for Mathematics and Interdisciplinary Sciences, Northeast Normal University, Changchun, China. Email: gaopengjilindaxue@126.com} \and Liying Li \thanks{Department of Mathematics, Southern University of Science and Technology, Shenzhen, China. Email: lily@sustech.edu}}

\date{}

\maketitle

\begin{abstract}
  In this paper, we study the long-time behavior of the viscous conservation law equation on the whole real line, subject to a white-in-time force and arbitrarily small damping and viscosity constants. We establish polynomial mixing rates in two settings: (a) a general super-quadratic flux with hyperviscous dissipation, where the growth of the nonlinearity is controlled by the order of hyperviscosity, and the force is sufficiently non-degenerate on a generic~\(L^2\)-basis; and (b) a general quadratic flux with standard second-order dissipation, and the force is non-degenerate on some unconditional basis. Both settings include the classical Burgers-type nonlinearity. The proof is based on an abstract polynomial coupling criterion developed in~\cite{Gao2026PolynomialMixingWhiteforced} and a Foia\c{s}--Prodi estimate. 
\end{abstract}



\tableofcontents

\section{Introduction}\label{sec:introduction}
The viscous conservation law equation is an important class of partial differential equations (PDEs). It provides a fundamental and complete description of real compressible flows across all speed regimes. In this paper, we consider the (damped) viscous conservation law equation with random forcing in one dimension
\begin{equation}\label{eq:conservation-law}
u_t + a u + \nu (-\partial_{xx})^{\gamma/2} u + \bigl( F(u) \bigr)_x = f^{\omega}, \quad \nu,\gamma > 0, \ a \ge 0.
\end{equation}
Equation \cref{eq:conservation-law} captures two competing physical mechanisms: nonlinear convection that steepens flow gradients into shock waves or shear layers, and viscous dissipation that smooths these sharp discontinuities through internal friction and thermal conduction.
When the flux~\(F = u^2/2\) is quadratic, equation \cref{eq:conservation-law} reduces to the Burgers equation, which was introduced initially~\cite{Burger1940HydrodynamicsApplicationModel,Burger1974NonlinearDiffusionEquation} as a simplified fluid equation to study the turbulence.
The addition of a (small) dissipative term is motivated by the vanishing-viscosity selection principle, which states that the physically relevant solutions of the inviscid conservation law can be recovered by the solutions of the viscous equation in the~\(\nu \to 0\) limit.
The selection principle is expected to hold not only for the usual second derivative, but also for general dissipative term. Rigorous results have been established for fractional dissipation~\(\gamma \in (0,2]\)~\cite{Dronio2003VanishingNonlocalRegularization,Alibau2007EntropyFormulationFractal},
while higher-order and fractional dissipations~\(\gamma \ge 2\) extend this mechanism and arise naturally in the study of hyperviscous conservation laws and turbulence~\cite{Tadmor2004BurgersEquationVanishing,Banerj2019FractionalHyperviscosityInduced}.

The primary goal of this paper is to understand the long-time behavior of \cref{eq:conservation-law}, particularly the existence and uniqueness of stationary states and the rate of mixing. The study of this topic has several motivations. From a physical viewpoint, this is a natural model for non-equilibrium statistical dynamics, as the energy injected into the system by the force is balanced by the dissipation effect. With quadratic flux, the stochastic Burgers equation shares similar nonlinearity to the Navier--Stokes nonlinearity, and is one of the simplest models in which aspects of turbulence can be studied rigorously~\cite{BoritcKuksin2021OnedimensionalTurbulenceStochastic,KuksinK41TheoryTurbulence}. The stochastic PDE (SPDE) can also be viewed as an infinite-dimensional random dynamical system or Markov process, and it is interesting to see how dissipation, nonlinear transport, and stochastic forcing interact to determine its ergodic properties.

For the mixing property of \cref{eq:conservation-law} on the real line that is studied in this paper, we have to add a non-zero damping~\(a > 0\) to control the low-frequency mode; this is a unavoidable technical assumption since our coupling techniques depends on partial differential equation energy estimates. For similar problems on bounded domains, such damping effect is already provided by the boundedness of the domain itself, so one can take~\(a = 0\). If merely the existence and uniqueness of stationary solutions are concerned, then certain probabilistic methods can bypass the damping addition, as we will see below. For the long-term behavior of randomly forced viscous conservation laws, there are only a few results available. The invariant measure for such equations was studied in~\cite{debussche2015invariant,chen2019invariant,drivas2023invariant} and the references therein.
Compared with the above works, a notable feature of our contribution is that we are able to establish polynomial mixing on the whole line, a setting not covered in previous studies.

The ergodic proprieties of \cref{eq:conservation-law} and related equations have been extensively studied in the past thirty years. On~\(\mathbb{T}^d\), the results are usually stated for the stochastic Hamilton--Jacobi--Bellman (HJB) equation
\begin{equation}\label{eq:HJB}
\partial_t h + F ( \nabla h ) = \nu \Delta h + \tilde{f}^{\omega}.
\end{equation}
In dimension one, \cref{eq:HJB} is the integrated form of \cref{eq:conservation-law} with~\(a = 0\) and~\(\gamma = 2\), namely, if~\(h\) solves \cref{eq:HJB}, then~\(u = \partial_x h\) solves \cref{eq:conservation-law} with~\(f = \tilde{f}\). The flux function~\(F\) is the Hamiltonian in the HJB equation. In higher dimensions the resulting equation from~\(u = \nabla h\) only generalizes \cref{eq:conservation-law} if additionally~\(u\) and~\(f^{\omega}\) are gradient fields. In this setup, it has been proven that, under fairly general assumptions of the random forcing, in every dimension, for all convex Hamiltonians, for both positive viscosity and zero-viscosity, there is a unique stationary solution~\(u\) for every prescribed average velocity~\(v = \int_{\mathbb{T}^d} u \in \R^d \), a One Force --- One Solution principle holds, and the dynamics is exponentially mixing~\cite{Sinai1991TwoResultsConcerning,EKhaninMazelSinai2000InvariantMeasuresBurgers,IturriKhanin2003BurgersTurbulenceRandom,DirrSougan2005LargetimeBehaviorViscous,GomesIturriKhaninPadill2005ViscosityLimitStationary,BoritcKhanin2013HyperbolicityMinimizers1D,debussche2015invariant,KhaninZhang2017HyperbolicityMinimizersRegularity, Boritc2018ExponentialConvergenceStationary,IturriKhaninZhang2020ExponentialConvergenceSolutions}; moreover, the viscous stationary solutions will converge to the inviscid ones in the~\(\nu \to 0\) limit~\cite{EKhaninMazelSinai2000InvariantMeasuresBurgers,GomesIturriKhaninPadill2005ViscosityLimitStationary}.

On the whole space~\(\R^d\), the ergodic problems of \cref{eq:conservation-law} is much harder, as the compactness is lost in many aspects: the domains, the functional spaces, the embedding structures and so on. There are two kinds of forcing to consider: the first kind of forcing is to take~\(f^{\omega}(t,x)\) as a spatially homogeneous random processes in~\(x\), which is a direct generalization of the periodic setting; the second kind is to allow~\(f^{\omega}(t,x)\) to have suitable decays as~\(\lvert x \rvert \to \infty \), so the noise will provide certain compactification effects.

For undamped Burgers equation with homogeneous force on~\(\R\), there is an analogous picture to the periodic case: for every prescribed properly defined ``average velocity''~\(v\), there exists a unique stationary solutions. This was first established for inviscid Burgers in~\cite{bakhtin2014space,Bakhti2016InviscidBurgersEquation}, and then in the viscous case in~\cite{bakhtin2019thermodynamic}. The inviscid limit of the stationary solutions is proven in~\cite{BakhtiLi2018ZeroTemperatureLimit}.
The exact quadratic form of~\(F\) of the Burgers equation allows reformulation of the ergodic problems in terms of infinite geodesics or infinite-volume polymer measures, and various probabilistic estimates then apply. Similar results are also established in~\cite{dunlap2021stationary} using more a PDE-based approach. Later~\cite{drivas2023invariant} shows existence of stationary solutions for at least some sparse set of velocity. It is worth noting that in all these results, the uniqueness part comes purely from the comparison principle and the ordering of one-dimensional space, and thus cannot be extended to higher dimensions.
Another motivation to study the one-dimensional HJB equation with fast-decorrelated homogeneous force is that it is conjectured to belong to the Kardar–Parisi–Zhang (KPZ) universality class. This has been proven rigorously for many integrable models like the KPZ equation itself, i.e., \cref{eq:HJB} with quadratic~\(F\) and~\(\tilde{f}\) the space-time white noise~\cite{AmirCorwinQuaste2011ProbabilityDistributionFree,MatetsQuasteRemeni2021KPZFixedPoint}, while for general stochastic conservation laws it is wide open; we refer to~\cite{BakhtiKhanin2018GlobalSolutionsRandom} for more discussion in this direction.

For the undamped stochastic Burgers equation with decaying force, using probabilistic representation of the solutions, existence and uniqueness of stationary solutions are also established in some settings~\cite{HoangKhanin2003RandomBurgersEquation,BakhtiKhanin2010LocalizationPerronFrobeniusTheory,Bakhti2013BurgersEquationPoisson}. A exponential mixing rate is only established in~\cite{BakhtiKhanin2010LocalizationPerronFrobeniusTheory} where the forcing creates a deep trap near the origin.
The goal of this paper is to establish mixing rate of~\cref{eq:conservation-law} with more general~\(L^2\)-noises, following the coupling approach developed in~\cite{Shirik2008ExponentialMixingRandomly,KuksinShirik2012MathematicsTwoDimensionalTurbulence,Gao2026PolynomialMixingWhiteforced}.

Mixing rates of SPDEs on the whole space is generally difficult from the viewpoint of PDEs. Even with a positive damping, one still does not have Poincar\'e inequality or compact Sobolev embedding which are essential for the parallel programs in bounded domains.
Recently,~\cite{NersesZhao2024ExponentialMixingWhiteForced} establishes exponential mixing for the complex Ginzburg--Landau equation on the real line by combining weighted Foia\c{s}--Prodi estimates with the classical coupling method, which provides a powerful methodological framework.
The work~\cite{Gao2026PolynomialMixingWhiteforced} shows polynomial mixing of \cref{eq:conservation-law} with~\(\gamma = 4\), by extending the coupling method and establishing an abstract criterion for polynomial mixing of Markov processes.
Here, the order of polynomial mixing is arbitrarily high, but the constant before the decay rate will depend on the order.
This framework is later applied to obtain polynomial mixing for the white-forced Navier-Stokes system in the whole space~\cite{NersesZhao2024PolynomialMixingWhiteforced} and the white-forced wave equation on the whole line~\cite{gao2024polynomialwave}. In all these results, the random force is assumed to take the form
\begin{equation}\label{eq:force-form}
f^{\omega}(t,x) = h + \eta, \quad \eta =  \frac{\partial}{\partial t}(\sum_{ j = 1}^{\infty} b_j e_j(x) \beta_j(t)),
\end{equation}
where~\( \{ e_j(x) \}_{j=1}^{\infty}\) is an orthonormal basis in~\(H = L^2 (\R)\) or~\(H^1(\R)\),~\(\{\beta_j(t)\}_{j=1}^{\infty}\) is a family of independent Wiener process, and~\(h \) is a fixed element in~\(H\) satisfying some controllability condition. The only non-degeneracy condition imposed is that~\(b_j \neq 0\) for~\(1 \le j \le N\), where~\(N\) depends on the basis and the number~\(a, \nu > 0\).

This paper extends and continues the line of research initiated in~\cite{Gao2026PolynomialMixingWhiteforced}. Our original motivation is to develop a weighted Foia\c{s}--Prodi estimate for the classical Burgers equation where~\(\gamma = 2\). However, the nonlinearity of the Burgers equation and general conservation law is worse than that of the Navier--Stokes equation, in the sense that there is no cancellation of convection term in the absence of divergence-free condition~\(\nabla \cdot u = 0\). We handle the nonlinearity in two approaches. The first approach is to use extra regularity provided by superviscosity~\(\gamma>2\). This generalizes~\cite{Gao2026PolynomialMixingWhiteforced} where~\(\gamma = 4\), and also allows the flux function to grow super-quadratically where the growth exponent is controlled by~\(\gamma\); see \cref{eq:H-noise}. The second approach is to keep the usual second order dissipation, but assume additionally that the~\(L^2\)-basis~\(\{ e_j \}_{j=1}^{\infty}\) appearing in \cref{eq:force-form} also forms \emph{an unconditional basis} in~\(H^s\) for some~\(s > 0\);  see \cref{eq:H-unconditional-basis}. We then use a new Poincar\'e inequality (see \cref{lem:poincare-cutoff-reverse}) to obtain a Foia\c{s}--Prodi estimate that can treat all flux functions with a most quadratic growth.

The remainder of this paper is organized as follows. In \cref{sec:assumpt-main-results}, we present the precise assumptions and main results of this paper. In \cref{sec:mixing-extension}, we first recall the well-posedness theory of~\cref{eq:conservation-law} via the variational framework, and then present the abstract criterion for polynomial mixing of Markov processes in~\cite{Gao2026PolynomialMixingWhiteforced} and sketch the construction of the coupling extension. In~\cref{sec:foia-prodi-estimate}, we derive weighted Foia\c{s}--Prodi estimates for~\cref{eq:conservation-law}, which are the key step in verifying the abstract criterion. In \cref{sec:energy-estimate} we establish several energy estimates and stopping time properties that will be used in the proof of the main theorem, which will be done in \cref{sec:mixing-proof}. In the appendix, we collect some functional inequalities and auxiliary estimates that will be used in this paper.

\subsection*{Notation}
We denote by~\(C, C_a, C_{\nu}\), etc, inessential positive constants that may change from line to line. We use notation~\({\lesssim, \lesssim_a, \lesssim_{\nu}}\), etc, to indicate inequalities valid up to an inessential multiplicative constants such as~\(C, C_a, C_{\nu}\).
We use~\(a\wedge b\) and~\(a \vee b\) to denote the minimum and the maximum of two numbers.

We write~\(H = L^{2}(\R)\) and denote its by~\(\lVert \cdot \rVert \), and inner product by~\(\langle \cdot,\cdot \rangle \). For~\(s > 0\),~\(H^s(\R) \) is the Sobolev space with norm
\begin{equation*}
\lVert f \rVert_s = \lVert  ( 1 + \lvert \xi \rvert^2  )^{s/2} \hat{f}(\xi)  \rVert.
\end{equation*}
We denote by~\(D^s\) the Fourier multiplier~\(\lvert \xi \rvert^s \), so that
\begin{equation*}
\lVert D^s f \rVert + \lVert f \rVert  \lesssim \lVert f \rVert_s, \quad s \ge 0.
\end{equation*}

Let \(X\) be a Polish space with a metric \(d_{X}\), the Borel \(\sigma\)-algebra on \(X\) is denoted by \(\mathcal{B}(X)\) and the set of Borel
probability measures by \(\mathcal{P}(X)\). The law of a random element~\(\xi\) is denoted by~\(\mathcal{L}(\xi)\).  We use~\(\ONE_{\Gamma}\) to denote the indicator function of an event~\(\Gamma\). We denote by~\(B_X(x,R)\) is the open ball in~\(X\) of radius~\(R>0\) centered at~\(x \in X\), and by~\(\bar{B}_X(x,R)\) its closure. \(C_{b}(X)\) is the space of bounded continuous functions \(f:X\rightarrow \mathbb{R}\) endowed with the
norm \(\|f\|_{\infty}=\sup_{u\in X}|f(u)|\).
We write \(C(X)\) when \(X\) is compact. \(L_{b}(X)\) is the space of functions~\(f\in C_{b}(X)\) such that
\begin{equation*}
 \|f\|_{L(X)}=\|f\|_{\infty}+\sup\limits_{u\neq v}\frac{|f(u)-f(v)|}{d_{X}(u,v)}<\infty.
\end{equation*}
The dual-Lipschitz metric on \(\mathcal{P}(X)\) is defined by
\begin{equation*}
\|\mu_{1}-\mu_{2}\|^{*}_{L(X)}=\sup_{\|f\|_{L(X)}\leq 1}|\langle f,\mu_{1}\rangle-\langle f,\mu_{2}\rangle|,\quad \mu_{1},\mu_{2}\in \mathcal{P}(X),
\end{equation*}
where \(\langle f,\mu \rangle=\int_{X}f(u)\mu(du)\). The total variation metric on \(\mathcal{P}(X)\) is defined by
\begin{equation*}
\|\mu_{1}-\mu_{2}\|_{var}=\frac12\sup_{f\in C_b(X),\|f\|_{\infty}\leq1}|\langle f,\mu_{1}\rangle-\langle f,\mu_{2}\rangle|,~~\mu_{1},\mu_{2}\in \mathcal{P}(X).
\end{equation*}

We use~\(\PP\) and~\(\EE\) to denote the probability measure and expectation on the space of SPDE solutions, and use~\(\Pp\),~\(\Ee\) to denote the probability and expectation on the coupling space. We use~\(\P\) and~\(\E\) for a generic probability space in the appendix. We use~\(\uu\) to denote a vector~\(\uu = (u, \tilde{u})\).

\bigskip
{\bf Acknowledgements.}
Peng Gao and Liying Li are grateful to Professors Konstantin Khanin and Vahagn Nersesyan for their valuable discussions on this problem.
Peng Gao thanks Professor Liying Li for his invitation to Southern University of Science and Technology, and for his hospitality.
Peng Gao thanks Professor Sergei Kuksin for his invitation to Universit\'{e} Paris Cit\'{e}, and for his scientific supervision on ergodicity and mixing for random dynamical systems. Peng Gao thanks the financial support of the China Scholarship Council (No. 202406620219).
This work is supported by National Natural Science Foundation of China (Grant No. 12371188).

\section{Main result}\label{sec:assumpt-main-results}

We consider the stochastic viscous conservation law
\begin{equation}\label{eq:sto-cons-law-assumption}
u_t + a u + \nu (-\partial_{xx})^{1 + \kappa} u + \bigl( F(u) \bigr)_x = h + \eta, \quad u(t,x) : \mathbb{R}^2 \to \mathbb{R},
\end{equation}
where \(\nu, a > 0, \ \kappa \ge 0\), \(\eta\) is given by \cref{eq:force-form}, \(\{ e_j(x) \}_{j=1}^{\infty}\) is an orthonormal basis in \(H := L^2(\mathbb{R})\), and \(\{\beta_j(t)\}_{j=1}^{\infty}\) is a family of independent Wiener processes defined on a filtered probability space \((\Omega, \mathcal{F}, \mathcal{F}_t, \mathbb{P})\) satisfying the usual conditions.
Without loss of generality, we can take
\begin{equation*}
\Omega :=  \mathcal{C}_0 (0, \infty; H)
\end{equation*}
to be the space of continuous function taking values in~\(H\) and vanishing at~\(t = 0\), and~\(\Fc\) is the completion of the Borel~\(\sigma\)-algebra associated with the local uniform convergence in this space. The operator~\((-\partial_{xx})^{s} = D^{2s}\) in \cref{eq:sto-cons-law-assumption} is the Fourier multiplier defined by
\begin{equation*}
(D^{2s} f)^{\wedge} (\xi) = \lvert \xi \rvert^{2s} \hat{f}(\xi), \quad 2s \ge 0.
\end{equation*}
When~\(\kappa = 0\) and~\(F(u) = u^2/2\), equation \cref{eq:sto-cons-law-assumption} reduces to the classical Burgers equation with damping:
\begin{equation}\label{eq:cburgers}
u_t + a u  + u u_x - \nu u_{xx} = h + \eta.
\end{equation}

We have the following assumptions. First, we assume that the growth and regularity of~\(F'\) is controlled by~\(\kappa_1\) as follows:
\begin{equation}\label{eq:H-growth-of-F}\tag{H1}
 F' \in \mathrm{Lip}_{1+ \kappa_1}, \quad  \kappa_1  \begin{cases}
= 0, & \kappa = 0, \\
\in [0,\kappa), & \kappa > 0,
\end{cases}
\end{equation}
where
\begin{equation*}
\mathrm{Lip}_{1+s} :=  \Bigl\{  G \text{ locally Lipschitz}: \sup_{ x } \frac{\lvert G(x) \rvert }{ \lvert x \rvert^{1+s} } + \sup_{x \neq y } \frac{ \lvert G'(x) - G'(y) \rvert }{ \lvert x-y \rvert^s } < \infty \Bigr\}, \quad s \ge 0.
\end{equation*}
Second, we assume that the noise is sufficiently regular and has mild decay at~\(\infty\):
\begin{equation}\label{eq:H-noise}\tag{H2}
\begin{split}
&e_j \in H^{1+\kappa}, \quad h\in H, \quad \varphi^{1+\kappa_1(2+\kappa)} h\in H,\\
&\mathcal{B}_0 := \sum_{j \ge 1} \lvert b_j \rvert^2 < \infty, \quad \mathcal{B}_1: = \sum_{j \ge 1} \lvert b_j \rvert^2 \lVert e_j \rVert_{1+\kappa}^2 < \infty, \quad \mathcal{B}_{\varphi,m}: = \sum_{j = 1}^{\infty} \lvert b_j \rvert^2 \lVert \varphi^m e_j \rVert^2< \infty,
\end{split}
\end{equation}
where~\(\varphi(x) = \ln (2+x^2)\) and~\(m \ge 1\) satisfies an explicit lower bound~\cref{eq:range-for-m0} depending on~\(\kappa_1, \kappa\); in particular, when~\(\kappa_1 = 0\), \(m = 1\) is sufficient.
Lastly, when~\(\kappa = 0\), we assume additionally that
\begin{equation}\label{eq:H-unconditional-basis}\tag{H3}
\{ e_j \}_{j \ge 1} \text{ is an unconditional basis for~\(H^{s_0}(\R) \) for some~\(s_0 \in (0,1/2) \)}.
\end{equation}

We recall the following definition and result about \emph{unconditional basis} used in this paper; see also~\cite[Chap 7-8]{Heil2011BasisTheoryPrimer}.
\begin{definition}\label{def:schauder-and-uncond-basis}
  Let~\(X\) be a Banach space and~\(\{ x_n \} \subset X\).
  \begin{enumerate}
    \item\label{item:6}~\(\{ x_n \}\) is a Schauder basis for~\(X\) if for every~\(x \in X\), there exist unique scalar~\(a_n(x)\) such that
          \begin{equation}\label{eq:31}
          x = \sum_n  a_n(x) x_n.
        \end{equation}
        \item\label{item:7}~\(\{ x_n \}\) is called an unconditional basis if the series in \cref{eq:31} converges unconditionally.
  \end{enumerate}
\end{definition}

\begin{proposition}\label{prop:unconditional-mapping}
  If~\(\{ e_n \}_{n=1}^{\infty}\) is an orthonormal basis of~\(L^2(\R)\) and an unconditional basis for~\(H^s(\R)\), then there exist weights~\(\{ w_n \}_{n=1}^{\infty}\) so that the map
  \begin{equation*}
  U: H^s(\R) \to \ell^2(w), \quad U f \mapsto \bigl( \langle f,e_n \rangle_{L^2}  \bigr)_{n=1}^{\infty},
\end{equation*}
is a Hilbert-space isomorphism, that is, for some~\(c, C > 0\),
\begin{equation*}
 c \sum_{n =1}^{\infty} \langle f,e_n \rangle^2  \le \lVert f  \rVert^2_{H^s(\R)} \le C \sum_{n =1}^{\infty} \langle f,e_n \rangle^2 .
\end{equation*}
\end{proposition}
A canonical example for~\(\{ e_n \}_{n=1}^{\infty}\) satisfying \cref{eq:H-unconditional-basis} are sufficiently smooth wavelet functions; see~\cite[Proposition 1, Chapter 6.4]{Meyer1992WaveletsOperators}.

Based on the well-posedness theory of stochastic viscous conservation law equation \cref{eq:conservation-law} in \cref{sec:well-posedness-spde}, the white-forced viscous conservation law equation \cref{eq:conservation-law} defines a Markov family \((u_{t}, \mathbb{P}_{u})\) parameterized by the initial condition \(u\in H\).
Let \(P_{t}(u,A):=\mathbb{P}(u_{t}(u,\omega)\in A)\) be the transition function. We introduce the Markov operators \(\mathcal{B}_{t}\) and \(\mathcal{B}_{t}^{*}\) associated with \(P_{t}\); see \cref{sec:AC} for the detailed definitions.

If \(\mu\in \mathcal{P}(H)\) such that \(\mathcal{B}_{t}^{*}\mu=\mu\) for any \(t\geq0\), the measure \(\mu\) is called an invariant measure for the Markov family \((u_{t}, \mathbb{P}_{u})\).
The invariant measure is a powerful tool for describing the long-term behavior of random dynamical systems.

The main result in this paper is the following.
\begin{theorem}\label{thm:main-theorem}
Let~\(a, \nu > 0\) and~\(\kappa \ge 0\). Assume \cref{eq:H-growth-of-F}, \cref{eq:H-noise} and additionally \cref{eq:H-unconditional-basis} if~\(\kappa = 0\). Then there exists an integer~\(N \ge 1\) such that as long as
\begin{equation*}
b_j \neq 0, \quad 1\le j \le N,
\end{equation*}
and
\begin{equation}\label{eq:span-assumption}
h \in \mathrm{span} \{ e_1, \dotsc, e_N \}
\end{equation} holds,
the white-forced viscous conservation law equation \cref{eq:conservation-law} admits a unique invariant measure \(\mu\in \mathcal{P}(H)\) with polynomial mixing rate~\((1+t)^{-p}\), with~\(p \in   (1, q^{\ast}/2-1) \). Here, the constant~\(q^{\ast} = q^{\ast}(\kappa_1, \kappa, m) \in    (4, \infty] \) is explicitly defined in \cref{eq:def-q-star}, equals to~\(\infty\) when~\(\kappa_1 = 0\), and grows linearly in~\(m\) for large~\(m\) when~\(\kappa_1 > 0\).
Moreover, for any~\(p \in   (1, q^{\ast}/2-1) \), there is a constant~\(C_p > 0\) such that
\begin{equation}\label{PMV}
\Big\lvert  \EE f \bigl( u(t) \bigr) - \int_H f (u) \, \mu(du)  \Big\rvert   \le C_p \lVert f \rVert_{\mathrm{Lip}} (1+t)^{-p} \bigl(  1  + \lVert u_0 \rVert^2  \bigr), \quad t \ge 0,
\end{equation}
for any initial data~\(u_0 \in H\) and any bounded Lipschitz-continuous function~\(f: H \to \R\).
\end{theorem}

\begin{remark}
The estimate \cref{PMV} is equivalent to the following estimate
\begin{equation*}
\begin{split}
\|P_{t}(u_{0},\cdot)-\mu\|_{L(H)}^{*}\leq C_p(1+\|u_{0}\|^{2})(1+t)^{-p},~t\geq 0,
\end{split}
\end{equation*}
and also to
\begin{equation*}
\begin{split}
\|\mathcal{B}_{t}^{*}\lambda-\mu\|_{L(H)}^{*}\leq C_p(1+\|u_{0}\|^{2})(1+t)^{-p},~t\geq 0,
\end{split}
\end{equation*}
for any \(\lambda\in \mathcal{P}(H)\).
\end{remark}

\cref{thm:main-theorem} provides a polynomial mixing result for the white-forced classical viscous Burgers equation~\cref{eq:cburgers}. For future reference, we state it explicitly as follows.
\begin{corollary}
Let~\(a, \nu > 0\). Assume \cref{eq:H-growth-of-F}, \cref{eq:H-noise} and \cref{eq:H-unconditional-basis}. Then, there exists an integer~\(N \ge 1\) such that as long as
\begin{equation*}
b_j \neq 0, \quad 1\le j \le N,
\end{equation*}
and \cref{eq:span-assumption} holds,
the white-forced classical viscous Burgers equation~\cref{eq:cburgers} admits a unique invariant measure~\({\mu\in \mathcal{P}(H)}\) with polynomial mixing rate~\((1+t)^{-p}\), for any~\(p > 1\). More precisely, for every~\(p > 1\), there exists a constant~\(C_p > 1\) such that
\begin{equation*}
\begin{split}
\Big\lvert  \EE f \bigl( u(t) \bigr) - \int_H f (u) \, \mu(du)  \Big\rvert   \le C_p \lVert f \rVert_{\mathrm{Lip}}  \bigl(  1  + \lVert u_0 \rVert^2  \bigr) (1+t)^{-p}, \quad t \ge 0, \quad u_0 \in H.
\end{split}
\end{equation*}
\end{corollary}

\section{Coupling method}\label{sec:mixing-extension}

\subsection{Well-posedness for \cref{eq:conservation-law}}\label{sec:well-posedness-spde}

In this section we sketch the well-posedness theory of \cref{eq:conservation-law}, based on the variational framework in~\cite{AgrestVeraar2024CriticalVariationalSetting}. The setting is simplified by the additive noise and polynomial nonlinearity.

Let~\({V = V_1 = H^{\gamma}}\) and~\({V^{\ast} = V_0 = H^{-\gamma}}\), where~\(\gamma = 1 + \kappa\), and~\(V_{\theta} = [V_0, V_1]_{\theta} \) be the interpolation space. In particular, \(H = H^{\ast} = L^2(\R) = V_{1/2}\).
 We rewrite \cref{eq:conservation-law} as
\begin{equation*}
d u(t) + A \bigl(t, u(t) \bigr) \, dt  = Q dW_t,
\end{equation*}
where~\(W\) is a cylindrical Wiener process on~\(H\), and
\begin{equation*}
Q: H \to H; \quad v \mapsto \sum_{j  = 1}^{\infty} b_j \langle v, e_j \rangle e_j
\end{equation*}
is a Hilbert–Schmidt operator due to the condition~\(\sum_{j = 1}^{\infty} b_j^2 < \infty\).
The operator~\(A\) takes the forms
\begin{equation*}
  A (t, u)  = A_0 u - \mathcal{F}(u) - h, \quad A_{0} u =\nu ( - \partial_{xx} )^{\gamma} u + a u , \quad \mathcal{F}(u) = - (F(u))_x,
\end{equation*}
where~\(A_0 \in L (V, V^{\ast})\),~\(\Fc: V \to V^{\ast}\),~\(h \in L^2 \subset V^{\ast}\).
Using the above notation, the following results is stated in~\cite[Section 3.1]{AgrestVeraar2024CriticalVariationalSetting}.
\begin{theorem}\label{thm:well-posedness}
  Suppose~\(A\) satisfies
  \begin{enumerate}
    \item (coercivity) there exists~\(\theta > 0\) such that
          \begin{equation*}
          \langle A_0 u, u  \rangle \ge \theta \lVert u \rVert^2_V;
        \end{equation*}
    \item\label{item:4} (boundedness of~\(A_0\)) there exists~\(C > 0\) such that
          \begin{equation*}
          \lVert A_0 w \rVert_{V^{\ast}}   \le C \lVert w \rVert_V;
        \end{equation*}
    \item\label{item:5} (Lipschitzness and local monotonicity) there exist~\(\rho \ge 0\) and~\(\beta \in (1/2,1)\) such that
          \begin{equation}\label{eq:index-in-monotonicity}
            2 \beta \le 1+ \frac{1}{ \rho + 1},
          \end{equation}
          so that
          \begin{equation}
          \label{eq:local-monotonicity}
            \lVert \mathcal{F}(u) - \mathcal{F}(v) \rVert_{V^{\ast}}    \lesssim  ( 1 + \lVert  u \rVert_{V_{\beta}}^{\rho} + \lVert v \rVert_{V_{\beta}}^{\rho}    ) \lVert  u - v \rVert_{V_{\beta}}.
          \end{equation}
  \end{enumerate}

  Then, for all~\(u_0 \in H\), there exists a unique (probability) strong solution~\(u \in \mathcal{C} \bigl( [0,\sigma); H \bigr) \cap L_{\mathrm{loc}}^2 \bigl( [0,\sigma); V \bigr)\), where~\(\sigma \in (0,\infty]\) is blow-up time, and
  \begin{equation*}
  \P \bigl(  \sigma < \infty, \ \sup_{ t \in [0,\sigma)}  \lVert u(t) \rVert^2_H + \int_0^{\sigma} \lVert u(t) \rVert_V^2 \, dt < \infty  \bigr) = 0.
  \end{equation*}
\end{theorem}

Under our setting, the blowup time is~\(\infty\) due to the a priori energy estimate that follows from \cref{sec:l2-energy-estimate}. The first and second assumptions in \cref{thm:well-posedness} are routine to check. We will here verify the third assumption, that is, identity~\(\rho\) and~\(\beta\) in \cref{eq:index-in-monotonicity,eq:local-monotonicity}.
Indeed,
\begin{align*}
\label{eq:30}
  \lVert \mathcal{F}(u) - \mathcal{F}(v) \rVert_{V^{\ast} } & = \sup_{ \lVert w \rVert_{H^{\gamma}} \le 1} \langle - \partial_x F(u) + \partial_x F(v), w  \rangle \\
                    & = \sup_{ \lVert w \rVert_{H^{\gamma}} \le 1 } \langle F(u) - F(v), \partial_x w \rangle \\
                    & \le \lVert F(u) - F(v) \rVert_{L^2} \\
                    & \lesssim ( 1 + \lVert u \rVert_{L^{\infty}}^{1+\kappa_1} + \lVert v \rVert_{L^{\infty}}^{1+\kappa_1}   ) \lVert u-v \rVert_{H^s} \\
                    & \lesssim ( 1+ \lVert u \rVert_{H^s}^{1+\kappa_1} + \lVert v \rVert_{H^s}^{1+\kappa}   ) \lVert u-v \rVert_{H^s},
\end{align*}
where~\(s > 1/2\). Comparing with \cref{eq:index-in-monotonicity,eq:local-monotonicity}, the parameter~\(s\) and~\(\kappa_1\) need to satisfy
\begin{equation*}
2 \beta \le 1 + \frac{1}{ 2+ \kappa_1}, \quad s =(1- \beta) \cdot (-\gamma)  + \beta \cdot \gamma = (2 \beta - 1) \gamma, \quad \beta \in (1/2, 1).
\end{equation*}
We can find~\(\beta\) (and hence~\(s\)) if and only if
\begin{equation*}
\frac{s}{\gamma} = 2 \beta - 1 \le \frac{1}{ 2 +\kappa_1} \iff s \le \frac{1+\kappa}{ 2 +\kappa_1}.
\end{equation*}
Since~\(s > 1/2\), this requires
\begin{equation*}
\frac{1}{ 2}< \frac{1+\kappa}{ 2 +\kappa_1}\iff \kappa_1 < 2\kappa.
\end{equation*}
Our growth condition \cref{eq:H-growth-of-F} is stronger than this. Due to the limitation of our coupling method, we cannot allow~\(\kappa_1 \ge \kappa\); 
see \cref{rmk:range-of-kappa}.
\subsection{An abstract coupling criterion}\label{sec:AC}

In order to prove \cref{thm:main-theorem}, we recall a general criterion from~\cite{Gao2026PolynomialMixingWhiteforced}, which gives sufficient conditions for polynomial mixing.
\par
Let \(X\) be a separable Banach space with a norm \(\|\cdot\|\). Let \((u_{t}, \mathbb{P}_{u})\) be a Feller family of Markov processes in \(X\), \(P_{t}(u,A):=\mathbb{P}(u_{t}(u,\omega)\in A)\) is
the transition function. We introduce the following associated Markov operators
\begin{equation*}
\begin{split}
&\mathcal{B}_{t}: C_{b}(X)\rightarrow C_{b}(X),~\mathcal{B}_{t}f(u):=\int_{X}f(v)P_{t}(u,dv),~\forall f\in C_{b}(X),\\
&\mathcal{B}_{t}^{*}: \mathcal{P}(X)\rightarrow \mathcal{P}(X),~\mathcal{B}_{t}^{*}\lambda(A):=\int_{X}P_{t}(u,A)\lambda(du),~\forall \lambda\in \mathcal{P}(X).
\end{split}
\end{equation*}

\begin{definition} (See~\cite{Shirik2008ExponentialMixingRandomly,KuksinShirik2012MathematicsTwoDimensionalTurbulence})
Let \((\mathbf{u}_{t},\Pp_{\mathbf{u}})\) be a Markov family in \(X\times X\),
\((\mathbf{u}_{t},\Pp_{\mathbf{u}})\) is called an \textbf{extension} of \((u_{t}, \mathbb{P}_{u})\) if for any \(\mathbf{u}=(u,u^{\prime})\in X\times X\) the laws under \(\Pp_{\mathbf{u}}\) of processes \(\{\Pi_{1}\mathbf{u}_{t}\}_{t\geq0}\) and \(\{\Pi_{2}\mathbf{u}_{t}\}_{t\geq0}\) coincide with
those of \(\{u_{t}\}_{t\geq0}\) under \(\mathbb{P}_{u}\) and \(\mathbb{P}_{u^{\prime}}\) respectively, where \(\Pi_{1}\) and \(\Pi_{2}\) denote the
projections from \(X\times X\) to the first and second component.
\end{definition}

We are now in a position to present an abstract criterion for polynomial mixing, established in~\cite{Gao2026PolynomialMixingWhiteforced}; see~\cite{gao2024polynomialwave} for a more readily applicable version. This criterion is the main tool we use to prove~\cref{thm:main-theorem}.
\begin{theorem} (See~\cite[Theorem 1.2 and Remark 1.1]{Gao2026PolynomialMixingWhiteforced} or~\cite[Theorem 2.1]{gao2024polynomialwave})\label{THACPM}
Let \((u_{t},\mathbb{P}_{u})\) be a Feller family of Markov processes in \(X\).
Supposed that
\par
(1) There exists an extension \((\mathbf{u}_{t},\Pp_{\mathbf{u}})\), a stopping time \(\sigma\), a closed set \(B\subset X\), and an increasing
function \(g(r)\geq1\) of the variable \(r\geq0\) such that the following two properties
hold
\par
\textbf{Recurrence} There exists \(p_0>1\) such that
\begin{equation}\label{eq:abstract-recurrence}
\mathbb{E}_{\mathbf{u}} \tau_{\textbf{B}}^{p_0}\leq G(\mathbf{u})
\end{equation}
for all \(\mathbf{u}=(u,u^{\prime})\in \textbf{X}:=X\times X\),
where \(\tau_{\textbf{B}}=\tau_{\textbf{B}}(\textbf{u},\omega)=\inf\{t\geq0:\textbf{u}_{t}(\textbf{u},\omega)\in \textbf{B}:=B\times B\}\) and \(G(\mathbf{u}):=g(\|u\|)+g(\|u^{\prime}\|)\).
\par
\textbf{Polynomial squeezing} There exist positive constants \(c_{1},c_{2},c_3,c_4\)
such that, for any \(\mathbf{u}\in \mathbf{B}\), we have
\begin{equation}\label{eq:abstract-squeezing}
  \begin{split}
&\|u_{t}-u_{t}^{\prime}\|\leq c_{1}\|u-u^{\prime}\|(t+1)^{-p_0},~{\rm{for}}~0\leq t\leq \sigma,\\
&\mathbb{P}_{\mathbf{u}}(\sigma=\infty)\geq c_{2},\\
&\mathbb{E}_{\mathbf{u}}(\mathbb{I}_{\sigma<\infty}\sigma^{p_0})\leq c_3,\\
&\mathbb{E}_{\mathbf{u}}(\mathbb{I}_{\sigma<\infty}G(\mathbf{u}_{\sigma})^{p_0})\leq c_4.
\end{split}
\end{equation}

(2) There is an increasing function \(\tilde{g}(r)\geq1\) such that
\begin{equation*}
\begin{split}
\mathbb{E}_{u}g(\|u_{t}\|)\leq \tilde{g}(\|u\|),~{\rm{for}}~t\geq 0,~u\in X.
\end{split}
\end{equation*}

Then the family \((u_{t},\mathbb{P}_{u})\) has a unique stationary measure \(\mu\in \mathcal{P}(X)\), and it holds for any \(p\le p_0\) that
\begin{equation*}
\begin{split}
\|P_{t}(u,\cdot)-\mu\|^{*}_{L}\leq \mathcal{V}(\|u\|)(t+1)^{-p},~{\rm{for}}~t\geq 0,~u\in X,
\end{split}
\end{equation*}
where \(\mathcal{V}\) is defined by \(\mathcal{V}(r):=C\bigl(g(r)+\tilde{g}(0)\bigr)\), with \(C>0\) a constant.
\end{theorem}

\subsection{Construction of an extension}\label{sec:constr-mixing-extens}
We outline how to construct an extension~\((\mathbf{u}_t, \Pp_{\mathbf{u}})\), this is inspired by the works~\cite{KuksinShirik2012MathematicsTwoDimensionalTurbulence,NersesZhao2024ExponentialMixingWhiteForced,Gao2026PolynomialMixingWhiteforced}.

Let~\(\{ u(t) \}_{t \ge 0}\) and~\(\{ u'(t) \}_{t \ge 0}\) be the solution to \cref{eq:conservation-law} with initial conditions~\(u,u' \in H\). Let~\(\{ v(t) \}_{t \ge 0}\) be the solution to the following auxillary process
\begin{equation}\label{eq:aux-process}
\left\{
\begin{aligned}
& v_t +  F(v)_x + a v + \mathsf{P}_N \bigl[  F(u)_x  - F(v)_x \bigr] = - \nu(-\partial_{xx})^{1+\kappa} v + h + \eta.  \\
&v \big\vert_{t = 0} = u'.
\end{aligned}\right.
\end{equation}
We fix a sufficiently large time step~\(\mathcal{T} > 0\) to be specified later. By~\cite[Theorem 1.2.28]{KuksinShirik2012MathematicsTwoDimensionalTurbulence}, there is a maximal coupling between the two laws on~\(\mathcal{C}(0,\mathcal{T}; H)\):
\begin{equation*}
\mathcal{L} \bigl( \{ v(t) \}_{t  \in [0, \mathcal{T} ]} \bigr), \quad \mathcal{L} \bigl(  \{ u'(t) \}_{t \in [0, \mathcal{T} ]} \bigr),
\end{equation*}
that is, there is some probability space~\((\tilde{\Omega}, \tilde{\Fc}, \tilde{\Pp})\), on which there are two~\(\mathcal{C}(0, \mathcal{T} ; H)\)-valued random elements~\(\mathfrak{U}'(\omega; u,u')\) and~\(\mathfrak{V}(\omega; u, u')\), such that
\begin{enumerate}
\item~\(\tilde{\Pp} \{  \mathfrak{U}' \in \cdot  \} = \PP \{ u' \in \cdot \}\), \(\tilde{\Pp} (\mathfrak{V} \in \cdot ) = \PP (v \in \cdot)\);
\item~\(\tilde{\Pp} \{  \mathfrak{U}' \ne \mathfrak{V} \} = \lVert \mathcal{L} ( \{ v(t) \}_{[0, \mathcal{T} ]} ) - \mathcal{L} ( \{ u'(t) \}_{[0, \mathcal{T} ]} ) \rVert_{TV}\);
\item conditioned on the event~\(\{  \mathfrak{U}' \neq \mathfrak{V} \}\), the two random elements~\(\mathfrak{U}',\mathfrak{V}\) are independent.
\end{enumerate}
Writing~\(\tilde{v}_t = \mathfrak{V}_t(\omega; u, u')\) and~\(\tilde{u}'_t = \mathfrak{U}'_t(\omega; u, u')\). Then the process~\(\{ \tilde{v}(t) \}_{t \in [0, \mathcal{T} ]}\) solves
\begin{equation*}
\left\{
\begin{aligned}
&\tilde{v}_t + F(\tilde{v})_x +  a\tilde{v} + \mathsf{P}_N \bigl(- F(\tilde{v})_x \bigr) = - \nu (-\partial_{xx})^{1+\kappa} \tilde{v} + h + \Lambda, \\
& \tilde{v} \big\vert_{t = 0} = u',
\end{aligned}\right.
\end{equation*}
where
\begin{equation*}
\{ \Lambda(t) \}_{t \in [0, \mathcal{T} ]}   \disteq  \bigl\{ \eta(t) - \mathsf{P}_N \bigl(F(u)_x\bigr) \bigr\}_{t \in [0, \mathcal{T} ] }.
\end{equation*}
Now let~\(\{ \tilde{u}(t) \}_{t \in [0, \mathcal{T} ]}\) solves
\begin{equation}\label{eq:couple-u}
\left\{
\begin{aligned}
&\tilde{u}_t + F(\tilde{u})_x + a \tilde{u} + \mathsf{P}_N [ - F(\tilde{u})_x] = - \nu (-\partial_{xx} )^{1+\kappa} \tilde{u} + h + \Lambda,\\
&\tilde{u} \big\vert_{t = 0} = u.
\end{aligned}\right.
\end{equation}
Note that~\(\{ u(t) \}_{t \in [0, \mathcal{T} ]}\) solves the same equation \cref{eq:couple-u} with~\(\tilde{u}\) replaced by~\(u\) and~\(\Lambda\) by~\(\eta - \mathsf{P}_N \bigl( F(u)_x \bigr)\), and hence by the uniqueness in law for \cref{eq:couple-u} (see \cref{sec:well-posedness-spde}), we have
\begin{equation*}
\{ \tilde{u}(t) \}_{[0, \mathcal{T} ]} \disteq \{ u(t) \}_{[0, \mathcal{T} ]}.
\end{equation*}
Therefore, on~\((\tilde{\Omega}, \tilde{\Fc}, \tilde{\Pp})\) there is a~\(\mathcal{C}(0, \mathcal{T} ; H)\)-valued random element~\(\mathfrak{U}(\omega; u,u')\) with distribution given by~\({\PP(u \in \cdot)}\), and~\((\mathfrak{U},\mathfrak{U}')\) is a coupling of~\(\bigl( u(t) \bigr)_{t \in [0, \mathcal{T} ]}\) and~\(\bigl( u'(t) \bigr)_{t \in [0, \mathcal{T} ]}\).

We now construct the mixing extension as follows. Let~\(\bigl\{  ( \Omega^k, \Fc^k, \tilde{\Pp}^k ) \bigr\}_{ k \ge 0}\) be independent copies of the space~\((\tilde{\Omega}, \tilde{\Fc}, \tilde{\Pp})\), and let~\((\Theta, \mathcal{G}, \Pp  )\) be their direct product.
The random elements~\(\mathfrak{U},\mathfrak{U}'\) can also be viewed as measurable maps defined on~\(\Omega^k\). For any~\(u, u' \in H\) and~\(\omega = (\omega^1, \omega^2, \dotsc ) \in \Theta\), let~\(\tilde{u}_0 = u\), \(\tilde{u}'_0 = u'\), and recursively
\begin{equation*}
\tilde{u}_t := \mathfrak{U}_s \bigl( \omega^k; \tilde{u}_{k \mathcal{T}}(\omega),  \tilde{u}'_{k \mathcal{T}} (\omega) \bigr), \quad
\tilde{u}'_t := \mathfrak{U}'_s \bigl( \omega^k; \tilde{u}_{k \mathcal{T}}(\omega),  \tilde{u}'_{k \mathcal{T}} (\omega) \bigr), \quad t =s + k \mathcal{T}, \ s \in [0, \mathcal{T} ), \ k \ge 0,
\end{equation*}

\medskip
The goal is to show that for suitably chosen~\(\mathcal{T}\), for~\(\mathbf{u}_t :=  (\tilde{u}_t, \tilde{u}'_t)\), the coupling \((\mathbf{u}_t, \Pp_{\mathbf{u}})\) is the desired mixing extension  of~\((u_t, \PP_u)\). The details will be given in \cref{sec:mixing-proof}. We briefly describe the mechanism below.

Due to the dissipation and damping, the Foia\c{s}--Prodi estimate guaranteed that~\(\lVert \tilde{u}(t) - \tilde{v}(t)  \rVert\) (which has the same law as~\(\lVert u(t) - v(t) \rVert\) under~\(\PP\)) decays exponentially, as long as certain energy of~\(\tilde{u}, \tilde{u}', \tilde{v}\) grow at most linearly. Thus, we introduce a stopping time~\(\sigma = \sigma_1 \wedge \tilde{\tau}\), where
\begin{align*}
 \sigma_1 &\approx \text{ the first time } \tilde{u}' \neq \tilde{v},  \\
\tilde{\tau} & \approx \text{ the first time energy of } \tilde{u}(t), \tilde{u}'(t), \tilde{v}(t) \text{ reach } C(1+t).
\end{align*}
When~\(t \le \sigma_1\), the Foia\c{s}--Prodi estimate will help to control~\(\lVert \tilde{u}(t) - \tilde{u}'(t) \rVert\), this is the first condition in polynomial squeezing.
Using a martingale inequality we can have power-law tail bound on~\(\tilde{\tau}\), and by Girsanov transform we can estimate the total variational distance~\(d_{TV} ( \tilde{v}(k \mathcal{T}), \tilde{u}'(k \mathcal{T}) )\), which leads to a power-law tail bound for~\(\sigma_1\). Combining these we obtain the remaining condition.

\section{Foia\c{s}--Prodi estimate}\label{sec:foia-prodi-estimate}

\subsection{Two Poincar\'{e}-type inequalities}

The first Poincaré-type inequality also appears in~\cite[Lemma 3.1]{NersesZhao2024ExponentialMixingWhiteForced}. We present a brief proof here for completeness.
\begin{lemma}\label{lem:poincare-cutoff}
Let~\(A > 0\). For every~\(\eps > 0\) and~\(s > 0\), there exists~\(N = N_{\eps}\) such that
\begin{equation*}
\lVert \mathsf{Q}_N \chi_A f \rVert \le\eps \lVert f \rVert_{H^s}.
\end{equation*}
\end{lemma}

\begin{proof}
It suffices to show that~\(T_N :=  \mathsf{Q}_N \circ \chi_A: H^s \to L^2\) converges strongly to~\(0\). The statement follows from the compactness of the operator~\(\chi_A: H^s  \to L^2\) by Rellich--Kondrachov theory, and that~\(\mathsf{Q}_N: L^2 \to L^2\) converges strongly to~\(0\) as~\(N \to \infty\).
\end{proof}

The second Poincaré-type inequality concerns~\(T^{\ast}_N = \chi_A \circ \mathsf{Q}_N\), which is only applicable to unconditional basis. This will be key step in the proof of the Foias\c{s}--Prodi estimate for \cref{eq:conservation-law}.
\begin{lemma}\label{lem:poincare-cutoff-reverse}
If~\(\{ e_n \}_{n=1}^{\infty}\) is an unconditional basis for~\(H^s(\R)(s>0)\), then there exists~\(N = N(\eps)\) such that
\begin{equation*}
\lVert \chi_A \mathsf{Q}_N f \rVert \le \eps \lVert f \rVert_{H^s}.
\end{equation*}
\end{lemma}

\begin{proof}
  Let~\(U\) and~\(\ell^2(w)\) be introduced in \cref{prop:unconditional-mapping}. Then
\begin{equation*}
\|f\|_{H^s}^2 \asymp \sum_{n=1}^\infty w_n |\langle f,e_n\rangle|^2 .
\end{equation*}
We can write~\(T^{\ast}_N = K \tilde{\mathsf{Q}}_N U\), where~\(K = \chi_A \circ U^{-1}: \ell^2(w) \to L^2(\R)\) is a compact operator, and~\(\tilde{\mathsf{Q}}_N = U \mathsf{Q}_N U^{-1}: \ell^2(w) \to \ell^2(w)\) is the truncation to the post-\(N\) coordinates.
It suffices to show the strong convergence
\begin{equation}\label{eq:op-norm-to-zero}
\lim_{N \to \infty} \lVert K \tilde{\mathsf{Q}}_N f \rVert_{L^2(\R)}  \to 0 , \quad f \in \ell^2(w).
\end{equation}
Indeed, let~\(\mathcal{S}\) be the set of operators~\(K: \ell^2(w) \to L^2(\R)\) such that \cref{eq:op-norm-to-zero} holds. We will show that
\begin{enumerate}
\item The set~\(\mathcal{S}\) is norm closed, that is, if~\(K_m \in \mathcal{S}\) and~\(K_m \to K\) in operator norm, then~\(K \in \mathcal{S}\).
\item If~\(K\) is finite-rank, then~\(K \in \mathcal{S}\).
\end{enumerate}
Then, since all compact operators between two Hilbert spaces can be approximated by finite-rank operators in the operator norm, \cref{eq:op-norm-to-zero} follows.

For the first item, the norm-closedness of \(\mathcal{S}\) follows from
\begin{equation*}
\lVert K \tilde{\mathsf{Q}}_N f \rVert_{L^2(\R)} \le \lVert K - K_m \rVert_{\ell^2(w) \to L^2(\R)} \lVert  f \rVert_{\ell^2(w)} + \lVert K_m \tilde{\mathsf{Q}}_N f \rVert_{\ell^2(w) \to L^2(\R)}.
\end{equation*}
For the second item, writing~\(c = (c_i) \in \ell^2(w)\), let~\(K\) be a finite-rank operator defined as
\begin{equation*}
K (c) :=  \sum_{ j = 1}^M \phi_j(c) y_j, \quad  \phi_j \in (\ell^2(w))^{\ast}, \ y_j \in L^2(\R).
\end{equation*}
Then~\(\phi_j\) can be identified as an element in~\(\ell^2(w^{-1})\):
\begin{equation*}
\phi_j(c) = \sum_{ k \ge 1} a^{(k)}_{j } c_k, \quad \sum_{k = 1}^{\infty} \lvert a^{(k)}_j \rvert^2/ w_j < \infty.
\end{equation*}
Hence,
\begin{equation*}
K\tilde{\mathsf{Q}}_N(c) = \sum_{j =1}^M y_j \sum_{ k \ge N } a^{(k)}_j c_k \le \sum_{j = 1}^M \lVert y_j \rVert_{L^2} \sum_{k \ge N} \lvert a^{(k)}_j \rvert^2 / w_k \to 0.
\end{equation*}
This completes the proof.
\end{proof}

\subsection{Foia\c{s}--Prodi estimate}\label{sec:FP-estimate}

Inspired by~\cite{abergel1989attractor,NersesZhao2024ExponentialMixingWhiteForced}, let us introduce a smooth space-time weight function.
Let~\(\varphi(x) = \ln (2+x^2)\) and
\begin{equation}\label{eq:def-psi}
\psi(t,x) = \varphi(x) \Bigl[  1 - \exp \Bigl( - \frac{t}{\varphi(x)} \Bigr) \Bigr].
\end{equation}
The weight function~\(\psi(t,x)\) has the following property:
\begin{enumerate}
\item~\(0 \le \psi (t,x) \le \varphi(x)\) for all~\(t \ge 0\), \(x > 0\)
\item~\(\psi(t,x) \to \infty\) for~\(t, x \to \infty\).
\item~\(\partial_t \psi, \partial_x^k \psi  \in L^{\infty} \) for~\(k \ge 1\) and~\( D^s \psi \in L^{\infty}\) for~\(s \ge 1\) (see also~\cref{lem:frac-deri-of-psi}).
\end{enumerate}

Let us consider the following systems
\begin{equation*}
\begin{split}
&u_t+au+\nu (-\partial_{xx})^{1+\kappa} u +\bigl( F(u) \bigr)_x =h+\eta,\\
&v_t+av+\nu(-\partial_{xx})^{1+\kappa} v+\mathsf{P}_N \bigl[F(u)_x-F(v)_x\bigr]+\bigl( F(v) \bigr)_x= h+\eta.
\end{split}
\end{equation*}
Set \(w=u-v\), then \(w\) satisfies
\begin{equation}\label{w-equ}
w_t + a w + Q_N \bigl[F(u)_x - F(v)_x \bigr] =
-\nu(-\partial_{xx})^{1+\kappa}w.
\end{equation}
We define
\begin{equation*}
e_z^{\alpha} (r) :=  \lVert z(r) \rVert^{2\alpha} \lVert z(r) \rVert_{1+\kappa}^2,
\quad \tilde{e}_z (r) :=  \lVert  \psi(r) z(r) \rVert^{2\kappa_1} \lVert \psi(r) z(r) \rVert_{1+\kappa}^2
~~~{\rm{for}}~r\ge 0.
\end{equation*}

Now, we are in a position to state the Foia\c{s}--Prodi estimate, for which \cref{lem:poincare-cutoff-reverse} plays an important role in the proof.
\begin{proposition}\label{lem:FP-estimate}
  Let~\(C_{\ast} > 0\) be a sufficiently large depending on~\(a, \nu, \kappa_1, \kappa\).
For every~\(\eps>0\), there exists~\(N_{\eps} \ge 1\) and~\(T_{\eps} \ge 2\) such that
\begin{equation}\label{eq:FP-estimate-ineq}
  \lVert w(t+T) \rVert^2 \le \lVert w(s+T) \rVert^2 \exp \Bigl(  - a(t-s) +
  C_{\ast} \eps \int_{s+T}^{t+T}  \bigl[e_u^{\alpha}(r) + e_v^{\alpha}(r) + \tilde{e}_u(r) + \tilde{e}_v(r)\bigr] \, dr
  \Bigr), \quad 0 \le s < t
\end{equation}
holds for all~\(N \ge N_{\eps} \ge 1\) and~\(T \ge T_{\eps} \ge 2\),
where \(w\) is the solution to equation
\cref{w-equ} and
\begin{equation*}
\alpha=
\begin{cases}
1,        & \kappa= 0, \\
\kappa_1, & \kappa> 0.
\end{cases}
\end{equation*}

\end{proposition}

\begin{proof}
Since
\begin{equation*}
F(u) - F(v) =w \int_0^1 F' \bigl(v + \lambda(u-v)\bigr) \, d\lambda = w \int_0^1 F'(z_{\lambda}) \, d \lambda,
\end{equation*}
where \(z_{\lambda}:=v + \lambda(u-v)\), we have
\begin{equation}\label{eq:FP-RHS}
  \partial_t \lVert w \rVert^2 + 2a \lVert w \rVert^2 + 2 \nu \lVert D^{1+\kappa} w \rVert^2 =
\langle  - \mathsf{Q}_N  [F(u) - F(v)]_x, w \rangle  = -  \int_0^1  \langle \mathsf{Q}_N  \partial_x[ w  F'(z_{\lambda})], w \rangle  \, d \lambda.
\end{equation}
By Gronwall's inequality, it suffices to show that for every~\(\eps \in (0, a \wedge \nu)\), uniformly in~\(\lambda\),
\begin{equation}\label{eq:upb-RHS-FP}
\Big\lvert  \langle \mathsf{Q}_N \partial_x [ w F'(z_{\lambda})] \rangle  \Big\rvert \le \eps \lVert w \rVert_{1+\kappa}^2 +  C \eps \lVert w \rVert^2 \Bigl[e_u^{\alpha} + e_v^{\alpha} + \tilde{e}_u + \tilde{e}_v  \Bigr].
\end{equation}
In order to prove \cref{eq:upb-RHS-FP}, We will estimate the following two terms
\begin{equation}\label{eq:two-term-decomp}
 \langle \mathsf{Q}_N \bigl(w_x F'(z_{\lambda}) \bigr), w \rangle , \quad \langle \mathsf{Q}_N \bigl(w \partial_x F'(z_{\lambda})\bigr), w \rangle.
\end{equation}
In the rest of the proof, we will drop the~\(\lambda\) dependence for brevity.

For the first term in \cref{eq:two-term-decomp}, we have
\begin{equation*}
  \langle \mathsf{Q}_N (w_x F'(z)),w \rangle  = \langle \mathsf{Q}_N \chi_A (w_x F'(z)), w \rangle  + \langle \mathsf{Q}_N (1-\chi_A) (w_x F'(z)), w \rangle =:   I_{11} + I_{12}.
\end{equation*}
For $I_{11}$, when~\(\kappa > 0\), using \cref{lem:poincare-cutoff}, for sufficiently large~\(N = N(\eps,A)\) we have for arbitrarily small~\(\kappa'>0\),
\begin{equation*}
 I_{11}  \le  \lVert  w \rVert  \lVert \mathsf{Q}_N \chi_A (w_x F'(z)) \rVert \le \eps\lVert w \rVert \lVert w_x F'(z) \rVert_{\kappa'}.
\end{equation*}
Then, by \cref{lem:fractional-product,lem:Lip-sobolev-embedding} and~\(\lvert F''(z) \rvert \lesssim \lvert z \rvert^{\kappa_1}\), we have
\begin{equation*}
  \begin{split}
    \lVert w_x F'(z) \rVert_{\kappa'} & \lesssim \lVert w_x \rVert_{\kappa'} \lVert F'(z) \rVert_{\frac12 + 2 \kappa'} \lesssim \lVert w \rVert_{1+\kappa} \lVert z \rVert^{\kappa_1}_{L^{\infty}} \lVert z  \rVert_{\frac12 + 2\kappa'} \lesssim \lVert w \rVert_{1+\kappa} \lVert z \rVert^{\kappa_1+1}_{\frac12 + 2 \kappa'}  ,
  \end{split}
\end{equation*}
and hence
\begin{equation}\label{eq:I11-sigma-positive}
  \begin{split}
    I_{11} & \le  \eps \lVert w \rVert^2_{1+\kappa} + C \eps \lVert w \rVert^2  \lVert z \rVert^{2\kappa_1+2}_{\frac12 + 2 \kappa'} \\
           & \le  \eps \lVert w \rVert^2_{1+\kappa}  +  C \eps \lVert w \rVert^2 \bigl(  \lVert u \rVert_{\frac12 + 2\kappa'}^{2 \kappa_1+2}  + \lVert v \rVert_{\frac12 + 2 \kappa'}^{2 \kappa_1+2}   \bigr) \\
    &\le  \eps \lVert w \rVert^2_{1+\kappa} + C \eps \lVert w \rVert^2 \bigl(  e_u^{\kappa_1} + e_v^{\kappa_1} \bigr),
  \end{split}
\end{equation}
where the last inequality follows from \cref{lem:sobolev-interpolation} with~\(\kappa'\) sufficiently small and~\(1+\kappa_1 < 2+2\kappa\).
When~\(\kappa = 0\), by \cref{eq:H-unconditional-basis} and \cref{lem:poincare-cutoff-reverse} with~\(s = s_0\in(0,\frac{1}{2})\), we have
\begin{equation}\label{eq:I11-sigma-zero}
  \begin{aligned}
I_{11} &= \langle \mathsf{Q}_N \chi_A (w_x F'(z)), w \rangle = \langle w_x F'(z), \chi_A \mathsf{Q}_N w \rangle  \\
& \le \eps \lVert w \rVert_{s_0} \lVert w_x  \rVert \lVert F'(z) \rVert_{L^{\infty}}       \\
&\le \eps \lVert w \rVert_{s_0} \lVert w_x \rVert \lVert z \rVert_{L^{\infty}} \\
&\le \eps \lVert w \rVert_1^{1+s_0} \lVert w \rVert^{1-s_0} \lVert z \rVert_{1-s_0} \\
       & \le \eps \lVert w \rVert_1^2 + C \eps \lVert w \rVert^2 \lVert z \rVert_{1-s_0}^{\frac{2}{1-s_0}} \\
    & \le \eps \lVert w \rVert_1^2 + C \eps \lVert w \rVert^2 \Bigl(  \lVert u \rVert_{1-s_0}^{\frac{2}{1-s_0}} + \lVert v \rVert_{1-s_0}^{\frac{2}{1-s_0}}   \Bigr) \\
    & \le \eps \lVert w \rVert_1^2 + C \eps \lVert w \rVert^2 \Bigl(  \lVert u \rVert^{\frac{2s_0}{1-s_0}}\lVert u \rVert^{2}_{1} + \lVert v \rVert^{\frac{2s_0}{1-s_0}}\lVert v \rVert^{2}_{1}  \Bigr) ~~({\rm{Since~}}\frac{2s_0}{1-s_0}<2)\\
    & \le \eps \lVert w \rVert_1^2 + C \eps \lVert w \rVert^2 \Bigl(  \lVert u \rVert^{2}\lVert u \rVert^{2}_{1} +\lVert u \rVert^{2}_{1} + \lVert v \rVert^{2}\lVert v \rVert^{2}_{1}   +\lVert v \rVert^{2}_{1}\Bigr)\\
     & \le \eps \lVert w \rVert_1^2 + C \eps \lVert w \rVert^2 \Bigl(  \lVert u \rVert^{2}\lVert u \rVert^{2}_{1} +\lVert\psi u \rVert^{2}_{1} + \lVert v \rVert^{2}\lVert v \rVert^{2}_{1}   +\lVert \psi v \rVert^{2}_{1}\Bigr)\\
& \le \eps \lVert w \rVert_1^2 + C \eps \lVert w \rVert^2  \bigl( e_u^1 + e_v^1 +\tilde{e}_u + \tilde{e}_v\bigr),
  \end{aligned}
\end{equation}
where the last second line we use~\(\psi(t,x) \ge c > 0\) for~\(t \ge 2\)
and
the last fourth line follows from standard interpolation or \cref{lem:sobolev-interpolation}.

For~\(I_{12}\), we have
\begin{equation*}
 I_{12}  \le \eps \lVert \psi F'(z) \rVert_{L^{\infty}} \lVert w_x  \rVert \lVert w \rVert  \le \eps \lVert w \rVert_1^2 + C \eps \lVert w \rVert^2  \lVert \psi F'(z) \rVert_{L^{\infty}}^2,
 \end{equation*}
where
\begin{equation*}
\lVert \psi F'(z) \rVert_{L^{\infty}} \lesssim
\begin{cases}
\lVert \psi z \rVert_{L^{\infty}} \lesssim \lVert \psi z \rVert_1,   & \kappa = 0,\\
\lVert \psi z^{1+\kappa_1} \rVert_{L^{\infty}}  \lesssim \lVert \psi z \rVert_{L^{\infty}}^{1+\kappa_1} \lesssim \lVert \psi u \rVert_{L^{\infty}}^{1+\kappa_1} + \lVert \psi v \rVert_{L^{\infty}}^{1 + \kappa_1} \lesssim (\tilde{e}_u + \tilde{e}_v)^{1/2},    & \kappa > 0.
\end{cases}
\end{equation*}
Here, in the last line we use~\(\psi(t,x) \ge c > 0\) for~\(t \ge 2\) in the first inequality, and in the last inequality
\begin{equation*}
\lVert \psi f \rVert_{L^{\infty} }^{1+\kappa_1} \lesssim \lVert \psi f \rVert_{\frac{1}{2} + s}^{1+\kappa_1} \le \lVert \psi f \rVert^{\kappa_1} \lVert \psi f \rVert_{1+\kappa}, \quad f = u, v,
\end{equation*}
which follows from \cref{lem:sobolev-interpolation} provided~\(s\) is arbitrarily small and~\(\frac12(1+\kappa_1) < 1 + \kappa\).

For the second term in \cref{eq:two-term-decomp}, we write
\begin{equation*}
  \langle \mathsf{Q}_N (w\partial_x F'(z)), w \rangle  =
\langle  \mathsf{Q}_N \chi_A (w \partial_x F'(z)), w \rangle  +   \langle \mathsf{Q}_N (1-\chi_A) (w \partial_x F'(z)), w  \rangle  =:   I_{21} + I_{22}.
\end{equation*}
For $I_{21}$, when~\(\kappa > 0\), by Cauchy--Schwartz and \cref{lem:poincare-cutoff}, we have
\begin{equation*}
I_{21} \le \lVert w \rVert \lVert \mathsf{Q}_N \chi_A (w \partial_x F'(z))  \rVert \le \eps \lVert w \rVert   \lVert w \partial_x F'(z) \rVert_{\kappa'}.
\end{equation*}
Using \cref{lem:kato-ponce} and \cref{lem:composition-operator}, by choosing arbitrarily small~\(\kappa'>0\) we have
\begin{align*}
I_{21} &\le \eps \lVert w \rVert  \lVert F'(z) \rVert_{1+\kappa'} \bigl( \lVert w \rVert_{L^{\infty}} + \lVert D^{\kappa'} w \rVert_{L^{\infty}} \bigr) \\
       &\le \eps \lVert w \rVert_{1+\kappa}^2 + C\eps \lVert w \rVert^2 \lVert z \rVert_{L^{\infty}}^{2\kappa_1}  \lVert z \rVert_{1+\kappa'}^2,  \\
& \le \eps \lVert w \rVert_{1+\kappa}^2 + C \eps \lVert w \rVert^2 \bigl( \lVert u \rVert_{\frac12 + \kappa'}^{2 \kappa_1} + \lVert v \rVert_{\frac12 + \kappa'}^{2 \kappa_1}    \bigr) \bigl(  \lVert u \rVert_{1+\kappa'}^2 + \lVert v \rVert_{1+\kappa'}^2   \bigr).
\end{align*}
For~\(f = u, v\), by \cref{lem:sobolev-interpolation} and~\(\kappa_1 + 2 < 2(1+\kappa)\),
\begin{equation*}
\lVert f \rVert_{\frac12 + \kappa'}^{2 \kappa_1 } \lVert f  \rVert_{1+\kappa'}^2  \lesssim \lVert f \rVert^{2 \kappa_1} \lVert f \rVert^2_{1+\kappa} = e^{\kappa_1}_f
\end{equation*}
provided that~\(\kappa'\) is sufficiently small. For the cross terms, using \cref{lem:sobolev-interpolation} we have
\begin{equation}\label{eq:cross-term-est}
  \lVert u \rVert ^{2 \kappa_1 }_{\frac12 + \kappa'} \lVert v \rVert_{1 + \kappa'}^{2} \lesssim \lVert u \rVert^{2\kappa_1+2}_{\frac12 +  \kappa'} + \lVert v \rVert^{2\kappa_1 + 2}_{ 1 + \kappa'}
  \lesssim \lVert u \rVert^{2 \kappa_1} \lVert u \rVert_{1+\kappa}^2 + \lVert v \rVert^{2 \kappa_1} \lVert v \rVert_{ 1 +\kappa}^2 = e_u^{\kappa_1 } + e_v^{\kappa_1},
\end{equation}
where~\(\kappa'\) needs to chosen sufficiently small so that
\begin{equation}\label{eq:sharp-bound-for-kappa}
\bigl( \frac{1}{2} + \kappa' \bigr) ( 2\kappa_1 + 2 ) \le 2(1+\kappa), \quad (1+\kappa')(2\kappa_1 + 2) \le 2(1+\kappa),
\end{equation}
which is possible if~\(\kappa_1 \in [0, \kappa)\); the other cross term can be estimated similarly. In conclusion, we have shown that for~\(z = \lambda u + (1-\lambda) v\), uniformly for~\(\lambda \in [0,1]\),
\begin{equation}\label{eq:unbalance-product-term-est}
\lVert z \rVert_{L^{\infty}}^{ 2 \kappa_1} \lVert z \rVert_{1+\kappa'}^2 \lesssim \lVert u \rVert^{2\kappa_1} \lVert u \rVert_{1+\kappa}^2 + \lVert v \rVert^{2 \kappa_1} \lVert v \rVert_{1+\kappa}^2
\lesssim e_u^{\kappa_1 } + e_v^{\kappa_1},
\end{equation}
provided~\(\kappa_1 < \kappa\) and~\(\kappa'>0\) sufficiently small.

When~\(\kappa = 0\), using \cref{eq:H-unconditional-basis} and \cref{lem:poincare-cutoff-reverse} with~\(s = s_0\in(0,\frac{1}{2})\), we have
\begin{align*}
I_{21} & \le \lVert \chi_A \mathsf{Q}_N w \rVert \lVert w \partial_x F'(z) \rVert \\
& \le \eps \lVert w\rVert_{s_0} \lVert w \rVert_{1-s_0} \lVert F''(z) \partial_x z \rVert \\
& \le   \eps \lVert w \rVert \lVert w \rVert_1 \lVert z\rVert_1 \\
       & \le \eps \lVert w \rVert_1^2 + \frac{\eps}{4} \lVert w \rVert^2 \lVert z \rVert_1^2 \\
  & \le \eps \lVert w \rVert_1^2 + \frac{\eps}{2} \lVert w \rVert^2 ( \lVert u \rVert^2_1 + \lVert v \rVert_1^2    )  .
\end{align*}

For~\(I_{22}\), we have
\begin{equation*}
I_{22} \le \eps \lVert w \rVert \lVert w  \rVert_{L^{\infty}} \lVert \psi \partial_x F'(z) \rVert \le \eps \lVert w \rVert_{1+\kappa}^2 + \frac{\eps}{4}\lVert w \rVert^2 \lVert \psi \partial_x F'(z) \rVert^2.
\end{equation*}
We have
\begin{equation*}
\lVert \psi \partial_x F'(z) \rVert = \lVert \psi F''(z) z_x \rVert
\le \begin{cases}
\lVert \psi z_x \rVert \le \lVert \psi z \rVert_1 + \lVert \psi_x \rVert_{L^{\infty}}  \lVert z \rVert \lesssim \lVert \psi z \rVert_1 ,    & \kappa = 0, \\
\lVert \psi z_x \rVert \lVert z \rVert_{L^{\infty}}^{\kappa_1} \lesssim  \lVert \psi z \rVert^{\kappa_1}_{L^{\infty}}  \lVert \psi z \rVert_1 \lesssim (\tilde{e}_u + \tilde{e}_v)^{1/2}  ,   & \kappa > 0,
\end{cases}
\end{equation*}
where we use~\(\psi(t,x) \ge c > 0\) for~\(t \ge 2\), and \cref{eq:unbalance-product-term-est} for the case~\(\kappa > 0\).

The inequality \cref{eq:upb-RHS-FP} follows from the above estimates on~\(I_{11}\), \(I_{12}\), \(I_{21}\) and~\(I_{22}\). This completes the proof.
\end{proof}

\begin{remark}\label{rmk:range-of-kappa}
The condition~\(\kappa_1 < \kappa\) is used in \cref{eq:sharp-bound-for-kappa} for the estimation of~\(I_{21}\) and~\(I_{22} \) in the case~\(\kappa>0\).
\end{remark}

We have the following estimate that allows~\(s,t\) to be small, if we do not require an arbitrarily small prefactor~\(\eps > 0\) as in \cref{lem:FP-estimate}.
\begin{proposition}\label{lem:w-naive-growth-bd}
There exists some constant~\(C= C (a,\nu)\) such that for all~\(N \ge 1\) and~\(t > s \ge 0\),
\begin{equation*}
\lVert w(t) \rVert^2 \le \lVert w(s) \rVert^2  \exp \Bigl( - a(t-s) + C\int_s^t  \bigl[e^{\kappa_1}_u(r) + e^{\kappa_1}_v(r)  \bigr] \, dr     \Bigr).
\end{equation*}
\end{proposition}

\begin{proof}
Following th notation in the proof of \cref{lem:FP-estimate}, we have
\begin{equation*}
\begin{split}
  \big\lvert \langle \mathsf{Q}_N (w_x F'(z)), w \rangle  \big\rvert  & \le \lVert w_x F'(z) \rVert \lVert w \rVert \le    \lVert w_x \rVert \lVert w \rVert \lVert F'(z) \rVert_{L^{\infty}} \\
                                                & \lesssim \lVert w_x \rVert \lVert w \rVert \lVert z \rVert^{\kappa_1+1}_{L^{\infty}}   \le \frac{a \wedge \nu}{2} \lVert w_x \rVert^2 + C_{a,\nu} \lVert w \rVert^2 \lVert z \rVert^{2 \kappa_1+2}_{\frac12 + \kappa'} \\
  &\le \frac{a \wedge \nu}{ 2} \lVert w_x \rVert^2 + C_{a,\nu} \lVert w \rVert^2 \bigl(  \lVert u \rVert_{\frac12 + \kappa'}^{2 \kappa_1 + 2} + \lVert v \rVert_{\frac12 + \kappa'}^{2 \kappa_1 + 2}   \bigr) \\
  & \le \frac{a \wedge \nu}{ 2} \lVert w_x \rVert^2 + C_{a,\nu} \lVert w \rVert^2 \bigl(  \lVert u \rVert^{ 2 \kappa_1} \lVert u \rVert_{1+\kappa}^2 + \lVert v \rVert^{2 \kappa_1} \lVert v \rVert_{1+\kappa}^2    \bigr), \\
  \big\lvert  \langle \mathsf{Q}_N (w F'(z)_x), w \rangle  \big\rvert & \le \lVert w F'(z)_x  \rVert \lVert w \rVert \le    \lVert w \rVert_{L^{\infty}} \lVert w \rVert \lVert F''(z) z_x \rVert \\
  & \lesssim \lVert w \rVert_{L^{\infty}} \lVert w \rVert \lVert z \rVert_{L^{\infty}}^{\kappa_1} \lVert z \rVert_1 \\
& \le \frac{a \wedge \nu}{2} \lVert w \rVert_{1+\kappa}^2 + C_{a,\nu} \lVert w \rVert^2 \bigl(  \lVert u \rVert^{2\kappa_1} \lVert u \rVert_{1+\kappa}^2 + \lVert v \rVert^{2 \kappa_1} \lVert v \rVert_{1+\kappa}^2     \bigr),
\end{split}
\end{equation*}
where the last line follows from \cref{eq:unbalance-product-term-est} if~\(\kappa>0\), or holds trivially if~\(\kappa_1 = \kappa = 0\). Hence, continuing \cref{eq:FP-RHS}, we have
\begin{equation*}
  \partial_t \lVert w \rVert^2 + 2a \lVert w \rVert^2 + 2 \nu \lVert D^{1+\kappa} w \rVert^2 \le
  (a \wedge \nu) \lVert w \rVert_{1+\kappa}^2 + C_{a,\nu} \lVert w \rVert^2 ( e^{\kappa_1}_u + e^{\kappa_2 } _v ).
\end{equation*}
The statement then follows from Gronwall's inequality.
\end{proof}

\section{Energy estimate}\label{sec:energy-estimate}

\subsection{\texorpdfstring{\(L^2\)}{L2}-energy estimate}\label{sec:l2-energy-estimate}
Let
\begin{equation*}
\mathcal{E}_u^p(t) :=  \lVert u(t) \rVert^{2p} + p(a \wedge \nu) \int_0^t\lVert u(s) \rVert^{2(p-1)} \lVert u(s) \rVert_{1+\kappa}^2 \, ds, \quad p \ge 1.
\end{equation*}
We have the~\(L^2\)-energy estimate.
\begin{lemma}\label{lem:L2-energy-estimate}
Let~\(q \ge 2\) and~\(p \ge 1\). For sufficiently large~\(K > 0\),
\begin{equation*}
\PP \Bigl(  \sup_{ t \ge 0} \bigl(   \mathcal{E}_u^p(t)  -  K(1+\|h\|^{2p}) t - K \lVert u_0 \rVert^{2p}   \bigr) \ge \rho \Bigr) \le  C_{p,q,K,D} \frac{\EE \lVert u_0 \rVert^{(2p-1)q} +\|h\|^{(2p-1)q}+ 1 }{  \rho^{q/2 } },\quad \rho \ge 2.
\end{equation*}
\end{lemma}

\begin{proof}
By Itô's formula, we have
\begin{equation}\label{eq:19}
  \begin{split}
 d \lVert u(t) \rVert^{2p} &= p \lVert u(t) \rVert^{2p-2} \Bigl(
2 \langle  u, -\nu D^{2+2\kappa} u  -  [F(u)]_{x} - au +h\rangle \, dt + 2 \langle u, d \eta \rangle + \mathcal{B}_0 dt \Bigr) \\
 & \mathrel{\phantom{=}}  +  2p (p-1) \lVert u \rVert^{2(p-2)} \sum_{j = 1}^{\infty} b_j^2 \langle u, e_j \rangle^2  \, dt \\
& \le - 2p \lVert u(t) \rVert^{2p-2} \bigl(a \lVert u(t)  \rVert^2  + \nu \lVert D^{1+\kappa} u(t) \rVert^2 \bigr) \, dt+2p\|u\|^{2p-1}\|h\| + d M_p(t) + K_p \mathcal{B}_0  \lVert u \rVert^{2(p-1)} \,  dt,
  \end{split}
\end{equation}
where
\begin{equation*}
M_p(t) := 2p  \int_0^t \lVert u(s) \rVert^{2(p-1)} \langle u(s), d\eta \rangle.
\end{equation*}
Integrating and using Young's inequality for~\(\lVert u \rVert^{2(p-1)} \le \eps \lVert u \rVert^{2p} + C\) and $\|u\|^{2p-1}\|h\|\leq \varepsilon\|u\|^{2p}+C\|h\|^{2p}$, we obtain
\begin{equation}\label{eq:6}
\mathcal{E}_u^p(t) -  C_{1,p} \lVert u_0 \rVert^{2p} \le C_{2,p} t +C_{3,p}\|h\|^{2p} t+C_{1,p} M_p(t),
\end{equation}
and hence
\begin{equation*}
\bigl\{  \mathcal{E}_u^p(t) - C_{1,p} \lVert u_0 \rVert^{2p} - (C_{1,p}+C_{2,p}+C_{3,p}\|h\|^{2p}) t \ge \rho   \bigr\} \subset \{ M_p(t) - t \ge \rho/C_{1,p} \}.
\end{equation*}
For~\(q \ge 2\), we have
\begin{equation*}
\begin{split}
\EE \langle M_p\rangle_t^{q/2} &= 4p^2 \EE  \Bigl[ \int_0^t \lVert u \rVert^{4 (p-1)} \sum_{j = 1}^{\infty} b_j^2 \langle u, e_j \rangle^2 \, ds   \Bigr]^{q/2}\\
& \lesssim_{p,q, \mathcal{B}_0} \EE \Bigl[  \int_0^t \lVert u \rVert^{4p-2} \, ds  \Bigr]^{q/2} \\
& \lesssim_{p,q,\mathcal{B}_0} t^{q/2-1} \EE \int_0^t \lVert u \rVert^{(2p-1)q} \, ds.
\end{split}
\end{equation*}
Integrating and taking expectation of \cref{eq:19} with~\(p\) replaced by~\((p-\frac12)q\), we obtain
\begin{equation*}
\EE \lVert u(t) \rVert^{(2p-1)q} \, \lesssim_{p,q,\mathcal{B}_0} \EE \lVert u_0 \rVert^{(2p-1)q} +\|h\|^{(2p-1)q}+ 1.
\end{equation*}
The desired inequality then follows from \cref{lem:martingale-linear-growth}.
\end{proof}

Integrating and taking expectation of \cref{eq:19}, then dropping the~\(\nu \lVert D^{1+\kappa} u \rVert^2 \) term and using the Young's inequality, we obtain
\begin{equation*}
\lVert u(t) \rVert^{2p} + pa \int_0^t \lVert u(s) \rVert^{2p} \, ds  \le C_{a,p,h,\mathcal{B}_0} t.
\end{equation*}
By Gronwall's inequality we obtain the following growth estimate.
\begin{lemma}\label{lem:L2-moment-growth}
  Let~\(u(t)\) be the solution to \cref{eq:conservation-law}. Then for any~\(p \ge 1\),
  \begin{equation*}
  \EE \lVert u(t) \rVert^{2p} \le e^{ - a pt}\EE \lVert u_0  \rVert^{2p} + C_{a,p,h,\mathcal{B}_0}.
  \end{equation*}
\end{lemma}

\subsection{Weighted \texorpdfstring{\(L^2\)}{L2}-estimate}\label{sec:weighted-l2-energy}
We define
\begin{equation*}
\mathcal{E}^{\psi}_u(t) :=  \lVert \psi u \rVert^{2(1+\kappa_1)} + (a  \wedge \nu) \int_0^t \lVert \psi u(s) \rVert^{2\kappa_1} \lVert \psi u(s) \rVert_{1+\kappa}^2  \, ds, \quad p \ge 1.
\end{equation*}
Assume~\(\mathcal{B}_{\varphi,m_0} < \infty \) where
\begin{equation}\label{eq:range-for-m0}
m_0 = m_0(\kappa_1, \kappa)  \begin{cases}
= 1, & \kappa_1 = 0, \\
> 1+2(\frac{1}{2} + \kappa_1 ) (2 + \kappa) =  \mathfrak{m}(1+\kappa_1,4) ,  & \kappa_1 >0;
\end{cases}
\end{equation}
see also \cref{eq:def-m-frak} for the definition of~\(\mathfrak{m}\).

The goal of this section is to establish the following estimate for weighted~\(L^2\)-energy. The lower bound for~\(m_0\) when~\(\kappa_1 > 0\) in~\cref{eq:range-for-m0} is to allow~\( q > 4\) that is needed to prove polynomial mixing in \cref{sec:poly-squeezing}; see also~\cref{eq:def-q-star} for the definition of~\(q^{\ast}\).
\begin{theorem}\label{lem:weighted-L2-energy-estimate}
  There exist~\(q_{\ast} = q_{\ast}(\kappa_1,\kappa,m_0) \in (4,\infty]\) and~\(\bar{p} = \bar{p}(\kappa,\kappa_1) \ge 1\) such that for every
\begin{equation}\label{eq:q-range}
2 \le q < q_{\ast},
\end{equation}
and sufficiently large~\(K\), 
\begin{equation}\label{eq:linear-growth-weight-L2}
\begin{split}
\PP \Bigl(  \sup_{ t \ge 0}\bigl(  \mathcal{E}^{\psi}_u(t) -  &K(1+\lVert h \rVert^{2 \bar{p}}+\lVert \varphi h \rVert^{2(1+\kappa_1)}+\lVert \varphi^{1+\kappa_1(2+\kappa)} h \rVert^2) t - K \lVert u_0 \rVert^{2 \bar{p}}\bigr) \ge \rho \Bigr) \\
&\le  C_{q,K} \frac{ \EE \lVert u_0 \rVert^{2 \bar{p} q} +\lVert h \rVert^{2 \bar{p}q}+\lVert \varphi^m h \rVert^{2} + 1}{  \rho^{q/2 } }, \quad \rho \ge 2,
\end{split}
\end{equation}
where $m= \mathfrak{m}(p,q) = 1 + q(p- \frac{1}{2})(2+\kappa)$.
\end{theorem}
The expressions of~\(q_{\ast}\) and~\(\bar{p}\) are explicit:
\begin{equation}\label{eq:def-q-star}
  q_{\ast} = \begin{cases}
+\infty, & \kappa_1 = 0, \\
  \sup \{ q: \mathfrak{m}(1+\kappa_1, q) \le m \} \sim  c_{\kappa_1,\kappa} m,   & \kappa_1 > 0,
\end{cases}
\end{equation}
where~\(\mathfrak{m}\) is given in \cref{eq:def-m-frak} or \cref{eq:range-for-m0}, and
\begin{equation}\label{eq:def-p-bar}
  \bar{p} =
  \begin{cases}
p_1 = 2 + \kappa, & \kappa_1 = 0,\\
 p_2 (1+\kappa_1) \le (2+2\kappa)(1+\kappa_1) + \kappa^{-1} + 2 , & \kappa_1 > 0,
\end{cases}
\end{equation}
(see also~\cref{eq:def-p1} and \cref{eq:def-p-2} for definitions of~\(p_1\) and~\(p_2\)).

To prove \cref{lem:weighted-L2-energy-estimate}, we introduce the following auxiliary energy function:
\begin{equation*}
\mathcal{G}_u^{p,m}(t) = \lVert \psi^m u(t) \rVert^{2p} + p(a\wedge \nu) \int_0^t \lVert \psi^m u \rVert^{2(p-1)}  \lVert \psi^m u(s) \rVert^2_{1+\kappa} \, ds, \quad p \ge 1, \ m \ge 1.
\end{equation*}
In the Itô's formula of~\(\lVert \psi^m u \rVert^2 \), we need to estimate
\begin{equation}\label{eq:def-A-u}
\mathcal{A}_m[u] = \langle \psi^m u, \psi^m ( \nu D^{2(1+\kappa) } u - F(u)_x + (\psi^m)_t u ) \rangle + \nu \lVert D^{1+\kappa}(\psi^m u) \rVert^2 .
\end{equation}

\begin{lemma}\label{lem:IBP-weight-Laplacian-term-alpha-1}
  Assume \cref{eq:H-growth-of-F}.
Let~\(m \ge 1\) and~\(\delta > 0\). There exists~\(C_{\delta,m} > 0\) such that
  \begin{equation*}
  \mathcal{A}_m[u] \le \delta \lVert \psi^m u \rVert^2_{1+\kappa} + C_{\delta,m}  \bigl( 1 +   \lVert u \rVert^{2(p_1 -1)} \lVert u \rVert_{1+\kappa}^2   \bigr) ,
\end{equation*}
where
\begin{equation}\label{eq:def-p1}
p_1 = 2 +  \kappa.
\end{equation}
\end{lemma}

\begin{proof}
  Let us write~\(\phi = \psi^m\).  Then by \cref{lem:frac-deri-of-psi} and~\(\lvert \partial_t \psi \rvert \lesssim 1 \), we have
  \begin{equation}\label{eq:22}
\lVert D^s \phi \rVert_{L^{\infty} } \le C, \ s \ge 1, \quad  \lvert \partial_t \phi \rvert \lesssim \lvert \psi \rvert^{m-1} = \lvert \phi \rvert^{\frac{m-1}{m}}.
\end{equation}
Let us treat the case~\(\kappa > 0\). We have
\begin{equation}\label{eq:25}
\begin{split}
\Big\lvert  \bigl(  D^{2(1+\kappa)} ( \phi u  ), \phi u \bigr) -  \bigl(  \phi u, \phi D^{2+2\kappa} u \bigr) \Big\rvert
& \le \Big\lvert \bigl(  D^{1+\kappa} (\phi^2 u) - \phi D^{1+\kappa} (\phi u), D^{1+\kappa} u \bigr)   \Big\rvert \\
&+ \Big\lvert \bigl(  D^{1+\kappa} (\phi u) - \phi D^{1+\kappa} u, D^{1+\kappa}(\phi u) \bigr)  \Big\rvert
\end{split}
\end{equation}
Using \cref{lem:commutator} and \cref{eq:22}, for~\(s \ge 1\), the first term is bounded by
\begin{equation}\label{eq:27}
C  \lVert D^{1+\kappa} u \rVert \Bigl[  \lVert \partial_x \phi \rVert_{L^{\infty}} \lVert D^{\kappa} (\phi u) \rVert + \lVert D^{1+\kappa} \phi \rVert_{L^{\infty} } \lVert \phi u \rVert     \Bigr]  \lesssim \lVert D^{1+\kappa} u \rVert \lVert \phi u \rVert_{\kappa} \le \delta \lVert  \phi u \rVert_{1+\kappa}^2 + C_{\delta} \lVert u \rVert_{1+\kappa}^2,
\end{equation}
and the second term is bounded by
\begin{equation}\label{eq:28}
C \lVert D^{1+\kappa} (\phi u) \rVert \Bigl[  \lVert \partial_x \phi \rVert_{L^{\infty}} \lVert D^{\kappa} u \rVert + \lVert D^{1+\kappa} \phi \rVert_{L^{\infty}} \lVert u \rVert     \Bigr]
\lesssim \lVert D^{1+\kappa} (\phi u) \rVert \lVert u \rVert_{\kappa} \le  \delta \lVert  \phi u \rVert_{1+\kappa}^2 + C_{\delta} \lVert u \rVert_{1+\kappa}^2.
\end{equation}
Let~\(\Phi(r) = \int_0^r s F'(s) \, ds\), we have
\begin{equation}\label{eq:26}
  \begin{split}
\Big\lvert  \langle \phi u, \phi F'(u) u_x \rangle  \Big\rvert &= \Big\lvert  \langle \phi^2 , [\Phi(u)]_x \rangle  \Big\rvert
                                     =  2 \lvert  \langle \phi \phi_x, \Phi(u) \rangle  \rvert \lesssim \lVert \phi u \rVert  \lVert u^{2 +\kappa_1} \rVert  \\
                                   &\le \delta \lVert \phi u \rVert^2_{1+\kappa} + C_{\delta} \lVert u \rVert^{4+2\kappa_1}_{L^{4+2\kappa_1}}\\
                                   &\le \delta \lVert \phi u \rVert^2_{1+\kappa} + C_{\delta} \lVert u \rVert^{4+2\kappa_1}_{\frac{1}{2} } \\
  & \le \lVert \phi u \rVert_{1+\kappa}^2 + C_{\delta} \lVert u \rVert^{2+2 \kappa_1} \lVert u \rVert_{1+\kappa}^2 ,
  \end{split}
\end{equation}
where we use~\( H^{1/2}\hookrightarrow  L^p\) and \cref{lem:sobolev-interpolation}.
Finally, by Young's inequality,
\begin{equation}\label{eq:time-derivative}
\big\lvert  \langle \phi u, \phi_t u  \rangle  \big\rvert \le\int_{\R} \Big\lvert  \phi^{\frac{2m-1}{2m}} u^2 \Big\rvert
  \le  \delta \lVert \phi u \rVert^2 + C_{\delta,m} \lVert u \rVert^2.
\end{equation}
The desired estimate follows from \cref{eq:27,eq:28,eq:26,eq:time-derivative}.

When~\(\kappa = 0\), the estimation of \cref{eq:25} is simpler since no fractional derivative is involved:
\begin{equation*}
\begin{split}
  \Big\lvert  \langle \partial_{xx} (\phi u), \phi u \rangle - \langle \phi u, \phi \partial_{xx} u \rangle   \Big\rvert
  & \le \lvert \langle \phi_{xx} u, \phi u \rangle  \rvert + 2 \lvert  \langle \phi_x u_x, \phi u \rangle  \rvert \\
  &\lesssim \delta \lVert \phi u \rVert_1^2 + C_{\delta} \lVert u \rVert_1^2.
\end{split}
\end{equation*}
We note that~\cref{eq:26,eq:time-derivative} remain valid at~\(\kappa = 0\) (so that~\(\kappa_1 = 0\) by \cref{eq:H-growth-of-F}). Combining this we prove the estimate for~\(\kappa = 0\).

\end{proof}

The following lemma proves \cref{lem:weighted-L2-energy-estimate} for~\(\kappa_1 =  0\).
\begin{lemma}\label{lem:weighted-energy-higher-m}
Let~\(m \ge 1\) and assume~\(\mathcal{B}_{\varphi,m} < \infty\). There exists~\(\mathcal{K}_m\) such that for~\(K \ge \mathcal{K}_m\),
  \begin{equation*}
    \PP \Bigl(  \sup_{ t \ge 0 } \bigl(  \mathcal{G}^{1,m}_u (t) - K(1+\|h\|^{2p_1}+\lVert \varphi^m h  \rVert^2)t -K \lVert u_0 \rVert^{2p_1}  \bigr) \ge \rho   \Bigr) \le C_{q,K} \frac{ \EE \lVert u_0 \rVert^{2p_1q} +\lVert h\rVert^{2p_1q} + 1 }{    \rho^{q/2 } }, \quad \rho \ge 2,
  \end{equation*}
  where~\(p_1\) is given in \cref{eq:def-p1}.
\end{lemma}

\begin{proof}
By Itô's formula, we have
  \begin{equation*}
  \begin{split}
      d \lVert\psi^m u(t) \rVert^{2}
= &2\mathcal{A}_m[u] \, dt + \Bigl[  - \nu \lVert D^{1+\kappa} (\psi^m u) \rVert^2  - a \lVert \psi^m u \rVert^2 +   \sum_{j = 1}^{\infty} b_j^2 \lVert \psi^m e_j \rVert^2 +2(\psi^m u,\psi^m h)  \Bigr] \,dt\\
      &+ d \tilde{M}_m(t),
      \end{split}
\end{equation*}
where
\begin{equation*}
\tilde{M}_m(t) :=   \int_0^t  \sum_{j = 1}^{\infty}   b_j \langle \psi^m u, \psi^m e_j \rangle  \, d \beta_j.
\end{equation*}
By \cref{lem:IBP-weight-Laplacian-term-alpha-1},  we have
\begin{equation*}
\begin{split}
  d \lVert \psi^m u(t) \rVert^2 \le & \Bigl[ - 2\nu \lVert D^{1+\kappa}(\psi^m u) \rVert^2   - a \lVert \psi^m u  \rVert^2 + C\lVert \psi^m h  \rVert^2 +\delta \lVert \psi^m u \rVert_{1+\kappa}^2 + C_{\delta,m} (\lVert u \rVert^{2(p_1 - 1)} \lVert u \rVert_{1+\kappa}^2 +1) \Bigr] \, dt
  \\&+ d \tilde{M}_m (t) + \mathcal{B}_{\varphi,m} \, dt.
\end{split}
\end{equation*}
Integrating the above inequality over~\([0,t]\) and choosing~\(\delta \le a \wedge \nu\), we obtain
\begin{equation}\label{eq:9}
   \mathcal{G}_u^{1,m} (t)  \le C_1(\mathcal{B}_{\varphi,m} + 1+\lVert \varphi^m h  \rVert^2) t + C_2 \mathcal{E}^{p_1}_u (t) +  \tilde{M}_m(t).
\end{equation}
Since
\begin{equation*}
  \langle \tilde{M}_m \rangle_{t} =  \int_0^{t} \sum_{j = 1}^{\infty} b_j^2 \langle \psi^m u, \psi^m e_j \rangle^2 \, ds
  \le \mathcal{B}_{\varphi,m} \int_0^{t} \lVert \psi^m u \rVert^2\, ds,
\end{equation*}
for sufficiently large~\(K'>0\) for \cref{lem:L2-energy-estimate} to apply and \(K> 0\), we have
\begin{multline*}
 \Bigl\{ \sup_{t \ge 0 } \bigl( \mathcal{G}_u^{1,m} (t)  - K(1+\|h\|^{2p_1}+\lVert\varphi^m h  \rVert^2) t - K \lVert u_0 \rVert^{2p_1} \bigr)  \ge \rho \Bigr\}  \\
\subset \Bigl\{  \sup_{t \ge 0 }\bigl(  \mathcal{E}^{p_1}_u  - K'(1+\|h\|^{2p_1} ) t - K' \lVert u_0 \rVert^{2p_1} \bigr)  \ge c_1 \rho \Bigr\}  \cup
\Bigl\{  \sup_{t \ge 0 }  \bigl(\tilde{M}_m(t) - \gamma \langle \tilde{M}_m \rangle_t  \bigr) \ge c_2\rho  \Bigr\},
\end{multline*}
for some constants~\(c_1, c_2, \gamma > 0\). The probability of the first event on the RHS can be bounded by \cref{lem:L2-energy-estimate}, and the probability of the second event by \cref{lem:super-mart-ineq}. This completes the proof.
\end{proof}

Next we need to treat the case~\(\kappa_1>0\). This means~\(\kappa>0\) and~\(\kappa_1 \in (0,2\kappa)\).
\begin{lemma}\label{lem:IBP-weighted-laplacian-term}
  Let~\(\alpha > 0\). There are constants~\(m = m_{\alpha} \ge 1, \beta = \beta_\alpha \ge 0 \) such that for all~\(\delta > 0\),
\begin{equation}\label{eq:10}
  \lVert \psi u \rVert^{2 \alpha} \mathcal{A}_1[u]
\le
\delta \lVert \psi u \rVert^{2\alpha}  \lVert  \psi u \rVert^2_{1+\kappa} + C_{\delta}  \Bigl( \lVert u \rVert^{2\beta}  \lVert u \rVert_{1+\kappa}^2  + \lVert \psi^m u \rVert^2 +1 \Bigr),
\end{equation}
where
\begin{equation}\label{eq:m-expre}
m = 1 + \alpha(2+\kappa),
\end{equation}
and
\begin{equation}\label{eq:beta-expre}
\beta = \Bigl( 1 + 2(\alpha+\kappa) \Bigr) \vee \frac{\alpha(1+\kappa)}{\kappa}.
\end{equation}
\end{lemma}

\begin{proof} Tracking \cref{eq:27,eq:28,eq:26,eq:time-derivative} for~\(m = 1\), we have
  \begin{equation}\label{eq:29}
  \lVert \psi u \rVert^{2 \alpha} \mathcal{A}_1 [u] \lesssim \lVert \psi u\rVert^{2 \alpha} \Bigl[  \lVert  u \rVert_{1+\kappa} \lVert \psi u \rVert_{\kappa} + \lVert  \psi u  \rVert_{1+\kappa} \lVert u \rVert_{\kappa} + \lVert \psi u \rVert \lVert u^{2+\kappa_1} \rVert    \Bigr]  =: J_1 + J_2 + J_3,
\end{equation}
where the LHS of \cref{eq:time-derivative} is bounded by~\(\lVert \psi u \rVert \lVert u \rVert  \) so the corresponding term can be bounded by either~\(J_1\) or~\(J_2\).

  For~\(J_1\), by \cref{lem:sobolev-interpolation} and Young's inequality, we have
\begin{equation*}
  \begin{split}
J_1 &\le \lVert \psi u \rVert^{2\alpha + \frac{1}{1+\kappa}} \lVert  u \rVert_{1+\kappa} \lVert \psi u \rVert_{1+\kappa}^{\frac{\kappa}{1+\kappa}}   \\
    &\le  \delta \lVert \psi u \rVert^{2\alpha} \lVert \psi u \rVert^2_{1+\kappa} + C_{\delta} \Bigl(  \lVert u \rVert^{2\beta_1} \lVert u \rVert^2_{1+\kappa} + \lVert \psi^{m_1} u \rVert^2    \Bigr),
  \end{split}
\end{equation*}
where
\begin{equation*}
m_1 = 1 + \alpha(2+\kappa), \quad \beta_1 = \frac{\alpha(2+\kappa)}{1+\kappa}.
\end{equation*}

For~\(J_2\), we have
\begin{equation*}
  \begin{split}
   J_2 & \le \lVert \psi u \rVert^{2\alpha} \lVert  \psi u \rVert_{1+\kappa} \lVert u \rVert^{\frac{1}{1+\kappa}} \lVert u \rVert^{\frac{\kappa}{1+\kappa}}_{1+\kappa} \\
    &\le \begin{cases}
\delta \lVert \psi u \rVert^{2\alpha} \lVert \psi u \rVert^2_{1+\kappa} + C_\delta \lVert u \rVert^{2\alpha} \lVert u \rVert^2_{1+\kappa},     & \alpha \le \kappa^{-1}, \\
 \delta \lVert \psi u \rVert^{2\alpha} \lVert  \psi u \rVert_{1+\kappa}^2 + C_{\delta} \Bigl(  \lVert u \rVert^{2 \beta_2} \lVert u \rVert_{1+\kappa}^2 + \lVert \psi^{m_2} u \rVert^2    \Bigr), & \alpha > \kappa^{-1},
\end{cases}
  \end{split}
\end{equation*}
where
\begin{equation*}
m_2 = \alpha(1+\kappa), \quad \beta_2 = \frac{\alpha(1+\kappa)}{\kappa}.
\end{equation*}

For~\(J_3\), using~\( H^{1/2} \hookrightarrow L^p \) and~\(\kappa_1 < 2\kappa\), we obtain
\begin{equation*}
J_3 \lesssim \lVert \psi u \rVert^{2\alpha+1} \lVert u \rVert^{2 + \kappa_1}_{\frac{1}{2}} \lesssim \lVert \psi u \rVert^{2 \alpha + 1} \bigl(  1 + \lVert u \rVert_{\frac{1}{2}}^{2 + 2 \kappa}  \bigr)
                    \lesssim \lVert u \rVert^{2 \beta_3} \lVert u \rVert_{1+\kappa}^2 + \lVert \psi^{m_3} u \rVert^2 + 1,
\end{equation*}
where
\begin{equation*}
m_3 = 1 + 2 \alpha, \quad \beta_3 = 1 + 2(\alpha+\kappa).
\end{equation*}

The lemma follows from these estimates and taking~\(m\),~\(\beta\) to be the maximum of~\(m_i\),~\(\beta_i\), at the cost of an additional additive constant on the RHS\@. This proves the lemma.
\end{proof}

Let
\begin{equation}\label{eq:def-p-2}
p_2 = p_2(\kappa) := \sup_{ \alpha > 0 } \frac{\beta_{\alpha} + 1}{ \alpha+1} \in (p_1, 2+2\kappa+\kappa^{-1}),
\end{equation}
where~\(\beta_{\alpha}\) is given by \cref{eq:beta-expre}.

The next lemma will be used to prove \cref{lem:weighted-L2-energy-estimate} for~\(\kappa_1 > 0\).
\begin{lemma}\label{lem:weigthed-L2-estimate-p-1}
  Let~\(q \ge  2, p > 1\) and assume that \(\mathcal{B}_{\varphi,m} < \infty\), where
\begin{equation}\label{eq:def-m-frak}
m= \mathfrak{m}(p,q) = 1 + q(p- \frac{1}{2})(2+\kappa) .
\end{equation}
There exists~\(\mathcal{K}_p > 0\) such that for~\(K \ge \mathcal{K}_p\), then
\begin{equation*}
\begin{split}
\PP \Bigl(  \sup_{ t \ge 0 }\bigl(  \mathcal{G}_u^{p,1}   - &K(1+\lVert h \rVert^{2p_2 p}+\lVert \varphi h \rVert^{2p}+\lVert \varphi^{m_1} h \rVert^{2})t - K\lVert u_0 \rVert^{2p_2 p }\bigr)\ge \rho    \Bigr) \\
&\le C_{q,K} \frac{ \EE \lVert u_0  \rVert^{2 p_2p q} + \lVert h  \rVert^{2 p_2p q}+\lVert \varphi^m h \rVert^{2} +1 }{ \rho^{q/2 }}, \quad \rho \ge 2,
\end{split}
\end{equation*}
where~\(m_1 = 1+(p-1)(2+\kappa)\). 
\end{lemma}

\begin{proof}
  By Itô's formula and a similar calculation to \cref{lem:L2-energy-estimate}, using \cref{lem:IBP-weighted-laplacian-term} for~\(\delta < (a \wedge \nu)/2\) as in the proof of~\cref{lem:weighted-energy-higher-m}, and noting \cref{eq:def-m-frak}, we have
  \begin{equation*}
  \begin{split}
    d \lVert \psi u \rVert^{2p} &=
                     p \lVert \psi u \rVert^{2p-2}
                     \Bigl( 2\bigl( \mathcal{A}_1[u ] - \nu \lVert D^{1+\kappa} (\psi u) \rVert^2 - a \lVert \psi u \rVert^2+\langle \psi^2 u, h\rangle  \bigr) \, dt+ \mathcal{B}_{\varphi,1}dt +   2 \langle \psi^2 u,d \eta \rangle   \Bigr) \\
    & \qquad + 2 p(p-1) \lVert \psi u \rVert^{2(p-2)} \sum_j b_j^2 \langle \psi u, \psi e_j \rangle^2 \\
&\le  \Bigl[    - p (a \wedge \nu) \lVert \psi u \rVert^{2p-2} \lVert \psi u \rVert_{1+\kappa}^2 + C_{p, \mathcal{B}_{\varphi,1}} \Bigl( \lVert u \rVert^{2(p_2 p -1)} \lVert u \rVert^2_{1+\kappa} + \lVert \psi^{m_1} u \rVert^2 +\lVert \varphi h \rVert^{2p}+ 1    \Bigr)     \Bigr] \, dt + \tilde{M}''_p(t),
  \end{split}
  \end{equation*}
and
  \begin{equation*}
  \tilde{M}_{p}''(t) = 2p \int_0^t \lVert \psi u \rVert^{2(p-1)}\sum_{j = 1}^{\infty} b_j \langle \psi u, \psi e_j  \rangle \, d\beta_j.
  \end{equation*}
Integrating over~\(t\), we obtain
  \begin{equation}\label{eq:8}
   \mathcal{G}^{p,1}_u (t) \le C_{p,\mathcal{B}_{\varphi,1}}(1+\lVert \varphi h \rVert^{2p}) t +  C_2\mathcal{E}^{p_2 p}_u (t) +  C_3 \mathcal{G}^{1,m_1}(t)  +  \tilde{M}''_p (t).
  \end{equation}
Hence, we have for sufficiently large~\(K \),~\(K_1\) and~\(K_2\),
\begin{equation*}
  \begin{split}
  & \Bigl\{ \sup_{t \ge 0 }\bigl(  \mathcal{G}^{p,1}_u(t) - K(1+\lVert h \rVert^{2p_2 p}+\lVert \varphi^{m_1} h \rVert^{2}+\lVert \varphi h \rVert^{2p}) t - K (\lVert u_0 \rVert^{2p_2 p}+\lVert u_0 \rVert^{2 p_1}) \bigr) \ge \rho \Bigr\}\\
    \subset  &\Bigl\{  \sup_{t \ge 0 }\bigl( \mathcal{E}^{p_2 p}_u(t) - K_1(1+\lVert h \rVert^{2p_2 p}) t - K_1 \lVert u_0 \rVert^{2p_2 p} \bigr)  \ge c_1\rho \Bigr\}  \\
    & \quad \cup \Bigl\{ \sup_{t \ge 0 } \bigl(  \mathcal{G}^{1,m_1} (t) - K_2(1+\lVert h \rVert^{2p_1}+\lVert \varphi^{m_1} h \rVert^{2}) t - K_2 \lVert u_0 \rVert^{2 p_1}  \bigr)  \ge c_2 \rho  \Bigr\}  \cup  \Bigl\{  \sup_{t \ge 0 } \bigl(\tilde{M}_p''(t)  - t\bigr) \ge  c_3\rho  \Bigr\},
  \end{split}
\end{equation*}
for some constants~\(c_i\); note that~\(p_2 p > p_2 > p_1\). The probability of the first and second events can be bounded by \cref{lem:L2-energy-estimate,lem:weighted-energy-higher-m}, respectively. It remains to show that
\begin{equation}\label{eq:12}
\PP \Bigl(  \sup_{t \ge 0 } \tilde{M}''_p(t) -  t \ge c_3 \rho  \Bigr) \lesssim_{\mathcal{B}_{\varphi,m}} \frac{ \EE \lVert u_0 \rVert^{2 p_2 p q} + 1 }{\rho^{q/2-1}}, \quad \rho \ge 2.
\end{equation}

We will prove \cref{eq:12} using \cref{lem:martingale-linear-growth} as in the last step of the proof of \cref{lem:weighted-L2-energy-estimate}.
First, we have
\begin{equation*}
    \EE \langle \tilde{M}_p'' \rangle_t^{q/2}  \lesssim_{p,q}t^{q/2-1} \EE \int_0^t \lVert \psi u \rVert^{(2p-1)q} \, ds.
\end{equation*}
Taking expectation of \cref{eq:8} with~\(p\) replaced by~\(p' = (p-1/2)q\) and~\(m_1\) replaced by~\(m = \mathfrak{m}(p,q)\), we obtain
\begin{equation*}
\EE \int_0^t \lVert \psi u \rVert^{(2p-1)q} \, ds \lesssim1+ t+ \EE \mathcal{E}^{p_2 p'}_u(t) + \EE \mathcal{G}^{1,m}_u(t).
\end{equation*}
Taking expectation of \cref{eq:6,eq:9}, respectively, we have
\begin{align*}
\EE \mathcal{E}^{p_2 p'}_u(t) \lesssim \EE \lVert u_0 \rVert^{2p_2 p'} +\lVert h \rVert^{2p_2 p'}t +  t, \quad \EE \mathcal{G}^{1,m}_u(t) \lesssim_{\mathcal{B}_{\varphi,m}} \EE \lVert u_0 \rVert^{2p_1}  +\lVert \varphi^m h \rVert^{2}t+\lVert h \rVert^{2p_1} t+  t.
\end{align*}
Finally we note that
\begin{equation*}
2p_2 p' =  p_2(2p-1)q \le 2 p p_2 q.
\end{equation*}
Combining these and \cref{lem:martingale-linear-growth} proves \cref{eq:12}.
\end{proof}

Based on the above lemmas, we can now prove \cref{lem:weighted-L2-energy-estimate}.

\begin{proof}[Proof of \cref{lem:weighted-L2-energy-estimate}]
The proof is divided into two cases.

  If~\(\kappa_1 = 0\), then~\(\mathcal{E}^{\psi}_u = \mathcal{G}_u^{1,1}\) and~\cref{lem:weighted-L2-energy-estimate} follows from \cref{lem:weighted-energy-higher-m}, which gives~\(\bar{p} = p_1\).

  If~\(\kappa_1 > 0\), then~\(\mathcal{E}^{\psi}_u = \mathcal{G}_u^{1+\kappa_1,1}\) and \cref{lem:weighted-L2-energy-estimate} follows from \cref{lem:weigthed-L2-estimate-p-1} with $p=1+\kappa_1$ which gives~\(\bar{p} = p_2(1+\kappa_1)\).
\end{proof}

\subsection{Stopping time estimate}\label{sec:stopping-time}
Motivated by the works~\cite{KuksinShirik2012MathematicsTwoDimensionalTurbulence,NersesZhao2024ExponentialMixingWhiteForced,Gao2026PolynomialMixingWhiteforced}, we introduce the following stopping times. For~\(K \ge 1\), we define
\begin{equation*}
\begin{split}
\tau^u_p  & := \inf \{  t \ge 0: \mathcal{E}_u^p(t)  \ge (K+1)t + K \lVert u_0 \rVert^{2p}  + \rho  \}, \quad p \ge 1, \\
\tau^u_{\psi} & := \inf \{  t \ge 0: \mathcal{E}_u^{\psi}(t) \ge (K+1) t + K \lVert u_0 \rVert^{2 \bar{p}} + \rho  \}.
\end{split}
\end{equation*}
where~\(\bar{p}\) is introduced in \cref{eq:def-p-bar}. Let us define
\begin{equation*}
 \tau^u := \tau^u_1 \wedge  \tau^u_{r} \wedge \tau^u_{\psi},
\end{equation*}
where
\(r = 2 \ONE_{\kappa = 0}+ (1+\kappa_1) \ONE_{\kappa>0} \).
For the stopping time \(\tau^u\), we have the following estimate.
\begin{lemma}\label{lem:growth-estimate-combined-energy}
Let~\(q\) be in \cref{eq:q-range}. If~\(K\) be sufficiently large, then
\begin{equation*}
\PP \bigl(  \ell \le \tau^u < \infty \bigr) \le C_{q,K} \frac{\EE \lVert u_0 \rVert^{2\bar{p} q}  + 1}{   (\rho + \ell)^{q/2 }} , \quad \rho \ge 2, \ \ell \ge 0.
\end{equation*}
\end{lemma}

\begin{proof}
We have the inclusion
\begin{align*}
\{ \ell \le \tau^u < \infty \} & \subset  \bigcup_{ p = 1, r} \Bigl\{ \exists t = \tau^u \ge \ell: \  \mathcal{E}^p(t) -  K t -  K  \lVert u_0\rVert^{2p} \ge \rho+\ell  \Bigr\} \\
& \quad \cup \Bigl\{\exists t = \tau^u\ge \ell: \ \mathcal{E}^{\psi}_u(t) - K t  -  K \lVert u_0 \rVert^{2\bar{p}}  \ge \rho+\ell \Bigr\}.
\end{align*}
The probability of the events on the RHS can be estimated by \cref{lem:L2-energy-estimate,lem:weighted-L2-energy-estimate}. The conclusion follows.
\end{proof}
\subsection{Estimate for auxiliary process}\label{sec:esti-aux}

The goal of this section is to give a growth estimate of~\(v\). The idea is to use the available estimate on~\(u'\) and use Girsanov theorem to control the effect of the finite dimension term
\begin{equation}\label{eq:finite-dim-drift}
\mathsf{P}_N  [ F(u)_x - F(v)_x].
\end{equation}
To view \cref{eq:finite-dim-drift} as a drift term of the noise and apply Girsanov, we need suitable truncation for the Novikov condition to hold. This motivates the definition of the truncated processes. Here we borrow some ideas from the works~\cite{NersesZhao2024ExponentialMixingWhiteForced,Gao2026PolynomialMixingWhiteforced,NersesZhao2024PolynomialMixingWhiteforced,gao2024polynomialwave}.

To be more precise, let
\begin{equation*}
\tau = \tau^v \wedge \tau^u \wedge \tau^{u'}.
\end{equation*}
We define~\(\hat{u}\) to be the truncated process of~\(u\), such that~\(\hat{u}(t) = u(t)\) when~\(t \le \tau\), and~\(\hat{u}\) solves the linear equation~\(\hat{u}_t + a \hat{u} = - (- \partial_{xx})^{1+\kappa} \hat{u}\) for~\(t \ge \tau\); the truncated processes~\(\hat{u}'\) and~\(\hat{v}\) are defined similarly. We can establish estimate for~\(\tau^{\hat{u}}\) and~\(\tau^{\hat{u}'}\) similar to \cref{lem:growth-estimate-combined-energy}. We recall that in \cref{lem:bound-tau-aux-process} and in the remaining of this section, the constant~\(\bar{p} = \bar{p}(\kappa_1,\kappa)\) is introduced in \cref{eq:def-p-bar}.
\begin{proposition}\label{lem:bound-tau-aux-process}
Let~\(q\) be in \cref{eq:q-range}. For sufficiently large~\(K\),
\begin{equation*}
\PP (\ell \le \tau^{\hat{u}} < \infty) \le C_{q} \frac{ \EE \lVert u_0 \rVert^{2\bar{p}q} + 1 }{    (\rho + \ell)^{q/2}}, \quad \rho \ge 2, \ \ell \ge 0.
\end{equation*}
\end{proposition}

\begin{proof} The proof of of \cref{lem:bound-tau-aux-process} is divided into several steps.

\textbf{Step 1.}
Let~\(z\) be the solution to the linear PDE
\begin{equation}\label{eq:linear-PDE}
z_t + a z = -(-\partial_{xx})^{1+\kappa} z.
\end{equation}
We have the following energy estimates. For~\(p, k \ge 1\),
  \begin{align}
  \label{eq:1}
    & \frac{d}{dt} \lVert z \rVert^{2p} + C_{a,\nu} \lVert z \rVert^{2(p-1)} \lVert z \rVert_{1+\kappa}^2  \le 0, \\
    \label{eq:18}
    & \frac{d}{dt} \lVert  \psi^k z \rVert^2 + C_{a,\nu,k} \bigl(  \lVert \psi^k z \rVert^2_{1+\kappa}  \bigr) \le C_{\nu,k} \lVert z \rVert^{2(p_1 -1)} \lVert u \rVert_{1+\kappa}^2 , \\
    \label{eq:20}
    & \frac{d}{dt} \lVert \psi z \rVert^{2p} + C_{a,\nu,p} \lVert \psi z \rVert^{2(p-1)} \lVert \psi z \rVert_{1+\kappa}^2
      \le C_{\nu,p} \bigl( \lVert u \rVert^{2\beta} \lVert u \rVert_{1+\kappa}^2 + \lVert \psi^m u \rVert^2 + 1    \bigr),
  \end{align}
  where~\(p_1\) is given in \cref{eq:def-p1}, and~\(\beta=\beta, m\) are given in \cref{lem:IBP-weighted-laplacian-term} with~\(\alpha = p-1\).

Indeed,
  \cref{eq:1} is essentially contained in the proof of \cref{lem:L2-energy-estimate}. Then \cref{eq:18} follows from \cref{lem:IBP-weight-Laplacian-term-alpha-1}, and \cref{eq:20} follows from \cref{lem:IBP-weighted-laplacian-term}.

\textbf{Step 2.} Let~\(u\) be the solution to \cref{eq:conservation-law} and~\(\hat{u}\) be the truncated process defined as in above with an arbitrary stopping time~\(\tau\).
Using \cref{eq:1}-\cref{eq:20}, we can compare the energy between the original process and the truncated process. More precisely, for~\(t \ge 0\):
  \begin{align}
\notag
    \mathcal{E}^p_{\hat{u}}(\tau + t ) & \le C_{a,\nu} \mathcal{E}^p_u(\tau), \\
    \label{eq:7}
    \mathcal{G}^{1,k}_{\hat{u}}(\tau + t) & \le C_{a,\nu,k} \mathcal{G}^{1,k}_u(\tau) + C_{\nu,k} \mathcal{E}^{p_1}_{u}(\tau), \\
    \notag
    \mathcal{G}^{p,1}_{\hat{u}}(\tau + t) & \le C_{a,\nu,p} \mathcal{G}^{p,1}_u(\tau)
                               + C_{\nu,p} \Bigl(  \mathcal{E}^{\beta+1}_u(\tau) + \mathcal{G}^{1,m}_u(\tau) + t \Bigr),
  \end{align}
  where the constants~\(C\),~\(p_1\),~\(\beta\) and~\(m\) are the same as in above arguments.

\textbf{Step 3. Proof of \cref{lem:bound-tau-aux-process}}

We first prove that for sufficiently large constant~\(K\), the statements in \cref{lem:L2-energy-estimate,lem:weigthed-L2-estimate-p-1,lem:weighted-energy-higher-m} hold with~\(u\) replaced by~\(\hat{u}\).

Indeed, we will only prove the probability bound for~\(\mathcal{G}^{1,m}_{\hat{u}}\) as an example. The other two statements are similar. By \cref{eq:7},  we have
  \begin{align*}
    \sup_{ t \ge 0 } \bigl(  \mathcal{G}^{1,m}_{\hat{u}}(t) - \tilde{K} t - \tilde{K} \lVert u_0 \rVert^{2p_1}  \bigr)
&    \le C_{a,\nu,m} \sup_{t \ge 0 } \bigl( \mathcal{G}^{1,m}_u(t) - K_1 t -  K_1\lVert  u_0 \rVert^{2 p_1}  \bigr) \\
& \qquad   + C_{\nu,m} \sup_{ t \ge 0 }  \bigl(  \mathcal{E}^{p_1}_u(t) - K_2 t - K_2 \lVert u_0 \rVert^{2 p_1}  \bigr),
  \end{align*}
  where~\(K_1 = \tilde{K}/4C_{a,\nu,m}\) and~\(K_2 = \tilde{K}/4C_{\nu,m}\), and they can be arbitrarily large if~\(\tilde{K}\) is large. The desired bound follows from \cref{lem:L2-energy-estimate,lem:weigthed-L2-estimate-p-1}.

Then, similarly to the derivation of \cref{lem:growth-estimate-combined-energy} from \cref{lem:L2-energy-estimate,lem:weighted-L2-energy-estimate}, \cref{lem:bound-tau-aux-process} follows immediately from the above results.
\end{proof}

We consider the transform on~\(\Omega = \mathcal{C}_0(0,\infty; H)\) defined by
\begin{equation*}
(\Phi^{u,u'} \omega)_t  = \omega_t - \int_0^t \ONE_{\{  s \le \tau \}} \mathsf{P}_N \bigl(  F(u)_x - F(v)_x \bigr) \, ds
\end{equation*}
Employing the strategy in~\cite[Section 3.3.3]{KuksinShirik2012MathematicsTwoDimensionalTurbulence}, we can estimate the total variation distance between~\(\PP\) and~\(\Phi^{u,u'}_{\ast} \PP\).
In this section we will write
\begin{equation*}
d = \lVert u - u' \rVert,  R = \lVert u \rVert \vee \lVert u' \rVert.
\end{equation*}
\begin{proposition}\label{lem:TV-distance-Phi}
Let~\(\tau\) be defined for sufficiently large~\(K\) and~\(\rho \ge 2\). Let~\(\epsilon > 0\). There exists~\(N \ge 1\) such that if~\(b_j \ne 0\) for~\(1 \le j \le N\), then for some constant~\(C\) depending on~\(N,K,\epsilon\) and~\(\max_{1 \le j \le N} b_j^{-1}\), such that
\begin{equation*}
\lVert \PP - \Phi^{u,u'}_{\ast} \PP \rVert_{TV} \le \frac{1}{2} \bigl[  \exp \bigl( d^2  e^{ C\rho + \epsilon R^{2\bar{p}} } \bigr) - 1 \bigr]^{1/2}.
\end{equation*}
\end{proposition}

\begin{proof}
Using the direct sum decomposition,
\begin{equation*}
\Omega = \mathcal{C}(0,\infty; \mathsf{P}_N H) \oplus \mathcal{C} (0, \infty; \mathsf{Q}_N H) :=  \Omega^{(1)} \oplus \Omega^{(2)},
\end{equation*}
we can rewrite~\(\Phi^{u,u'}\) as
\begin{equation*}
\Phi^{u,u'} (\omega^{(1)}, \omega^{(2)}) = \bigl( \Psi^{u,u'} (\omega^{(1)}, \omega^{(2)}), \omega^{(2)} \bigr),  \quad \omega = (\omega^{(1)}, \omega^{(2)}) \in \Omega,
\end{equation*}
where~\(\Psi^{u,u'}: \Omega^{(1)} \to \Omega^{(1)}\) is given by
\begin{equation*}
\Psi^{u,u'} (\omega^{(1)}, \omega^{(2)})_t :=  \omega_t^{(1)} + \int_0^t \mathcal{A} (s; \omega^{(1)}, \omega^{(2)}) \, ds, \quad \mathcal{A}(t) := - \ONE_{\{   t \le \tau \}} \mathsf{P}_N \bigl( F(u)_x - F(v)_x \bigr).
\end{equation*}
Writing~\(\PP = \PP_{N} \otimes \mathbb{Q}_N\), by~\cite[Lemma 3.3.13]{KuksinShirik2012MathematicsTwoDimensionalTurbulence}, we have
\begin{equation*}
\lVert \PP - \Phi^{u,u'}_{\ast} \PP \rVert_{TV} \le \int_{ \Omega^{(2)}} \lVert \Psi^{u,u'}_{\ast} (\PP_N, \omega^{(2)}) - \PP_N \rVert_{TV} \mathbb{Q}_N (d \omega^{(2)}).
\end{equation*}
The transform~\(\Psi^{u,u'}\) is acting on a finite dimensional Wiener process, and hence by the Girsanov theorem, for each~\(\omega^{(2)}\), we have
\begin{equation}\label{eq:2}
\lVert \Psi^{u,u'}_{\ast} ( \PP_N, \omega^{(2)} ) - \PP_N \rVert_{TV} \le \frac{1}{2} \Bigl[  \Bigl(  \EE_N \exp \bigl(  6 \sup_{1 \le j \le N}  b_j^{-2}  \int_0^{\infty} \lVert \mathcal{A} (t; \cdot , \omega^{(2)})\rVert^2 \, dt   \bigr) \Bigr)^{1/2} - 1 \Bigr]^{1/2},
\end{equation}
provided that the Novikov condition
\begin{equation}\label{eq:novikov}
\EE_N \exp \Bigl(  c \int_0^{\infty} \lVert \mathcal{A}(t; \cdot, \omega^{(2)}) \rVert^2 \, dt  \Bigr) < \infty
\end{equation}
is satisfied for some~\(c > 0\).

Next we will bound~\(\int_0^{\infty} \lVert \mathcal{A} \rVert^2\) which leads to \cref{eq:novikov} and the desire bound in the statement of the lemma. Recall that~\(w = u - v\). We have
\begin{equation}\label{eq:14}
\begin{split}
\lVert \mathcal{A} (t)\rVert^2 =  \ONE_{\{ t \le \tau \}} \lVert \mathsf{P}_N ( F(u)_x - F(v)_x )  \rVert
& \le \ONE_{\{  t \le \tau \}} \sum_{j = 1}^N \langle  w \int_0^1 F'(z_{\lambda}) \, d \lambda,  \partial_x e_j \rangle^2   \\
& \le \ONE_{\{  t \le \tau \}} \sum_{j = 1}^N \lVert w(t) \rVert^2 ( \lVert u(t) \rVert_{1+\kappa}^2  + \lVert v(t) \rVert_{1+\kappa}^2  ) \lVert e_j \rVert_1^2 \\
& \lesssim_N \ONE_{\{  t \le \tau \}} \lVert w(t) \rVert^2    ( \lVert u(t) \rVert_{1+\kappa}^2  + \lVert v(t) \rVert_{1+\kappa}^2  ).
\end{split}
\end{equation}

We first show exponential decay of~\(\lVert w(t) \rVert \), here we borrow some ideas from the works~\cite{NersesZhao2024ExponentialMixingWhiteForced,Gao2026PolynomialMixingWhiteforced,NersesZhao2024PolynomialMixingWhiteforced,gao2024polynomialwave}. Indeed, by \cref{lem:FP-estimate}, for
\begin{equation}\label{eq:def-var-epsilon}
\eps = \frac{ a \wedge \epsilon}{ 16 C_{\ast} K}.
\end{equation}
there are~\(T \ge 2\) and~\(N \ge 1\) so that \cref{eq:FP-estimate-ineq} holds.
If~\(t\le T \wedge \tau\), then by \cref{lem:w-naive-growth-bd} and we have
\begin{equation*}
\lVert w(t) \rVert^2 \le d^2  \exp \Bigl( - a t + C_{a,\nu} \int_0^t \bigl[ e^{\kappa_1}_u(s) + e^{\kappa_1}_v(s) \bigr] \, ds \, dr      \Bigr),
\end{equation*}
where by the definition of~\(\tau\), the integral in the exponential can be further bounded by
\begin{equation*}
  \mathcal{E}_u^{1+\kappa_1 } (t ) + \mathcal{E}_v^{1+\kappa_1 } (t) \le 2 \rho + 2(K+1) t + K R^{2(1+\kappa_1)}
  \le 2 \rho + 2(K+1) t +  KC_{\delta}   + K \delta R^{2\bar{p}},
\end{equation*}
where we use the Young's inequality~\(R^{2(1+\kappa_1)} \le C_{\delta}+ \delta R^{2\bar{p}}\).
By choosing~\(\delta = \frac{\epsilon}{4K C_{a,\nu}}\), we obtain
\begin{equation}\label{eq:tau-small}
\lVert w(t) \rVert^2 \le d^2 \exp \Bigl(  -a t + C_{K,\epsilon} (\rho+t) + \frac{\epsilon}{4} R^{2\bar{p}} \Bigr), \quad t \le \tau \wedge T.
\end{equation}
When~\(t \in [T, \tau]\), the definition of~\(\tau\), we have
\begin{equation*}
  \int_T^t \bigl[  e^{\alpha}_u(s) + e^{\alpha}_v(s) + \tilde{e}_u(s) + \tilde{e}_v(s) \bigr] \, ds
  \le 4 \rho +  4 (K+1) t + K R_r + K R^{2\bar{p}} \le 8K(\rho +t )+ K R^{2\bar{p}},
\end{equation*}
and hence by \cref{lem:FP-estimate}, \cref{eq:tau-small} at~\(t = T\) and our choice of~\(\eps\) in \cref{eq:def-var-epsilon}, we have
, we have
\begin{equation}\label{eq:L2-norm-decay-under-coupling}
\begin{split}
  \lVert w(t) \rVert^2 & \le d^2 \exp \Bigl(  -a T + C_{K,\epsilon} (\rho+T) + \frac{\epsilon}{2} R^{2\bar{p}} \Bigr) \times
               \exp \Bigl(  - a (t-T) + C_{\ast} \eps \bigl(  8K(\rho+t) + 2K R^{2\bar{p}}  \bigr) \Bigr)
  \\
  & \le d^2 \exp \Bigl(  - \frac{a}{2} t + C_{K,\epsilon} \rho + \epsilon R^{2\bar{p}}/2 \Bigr).
\end{split}
\end{equation}
By adjusting the constant~\(C_{K,\epsilon}\), the inequality \cref{eq:L2-norm-decay-under-coupling} still holds for~\(t \le T \wedge \tau\).

Let
\begin{equation*}
H(t) = \int_0^t( \lVert u(s) \rVert^2_{1+\kappa} + \lVert v(s) \rVert^2_{1+\kappa} )   \, ds.
\end{equation*}
Then~\(H \ge 0\) and
\begin{equation}\label{eq:15}
H(\tau) \le \mathcal{E}^1_u(\tau) + \mathcal{E}^1_v(\tau) \le 2\rho + (K+1) \tau + KR_1 \le C_{K,\epsilon}(\rho + \tau) +  \epsilon R^{2\bar{p}}/2.
\end{equation}
By \cref{eq:14,eq:L2-norm-decay-under-coupling,eq:15}, we have
\begin{align*}
  \int_0^{\infty}\lVert \mathcal{A}  \rVert^2  \, dt & \lesssim_N \int_0^{\tau} \lVert w(t) \rVert^2  \, dH(t) \le d^2 e^{C_{K,\epsilon} \rho + \epsilon R^{2\bar{p}}/2} \int_0^{\tau} e^{ - \frac{a}{2}t}\, dH(t) \\
                         &\le d^2 e^{C_{K,\epsilon} \rho + \epsilon R^{2\bar{p}}/2}  H(\tau) e^{ - a\tau/2}  \\
  & \le d^{2} e^{ C_{K,\epsilon} \rho + \epsilon R^{2\bar{p}}}.
\end{align*}
Then the Novikov condition \cref{eq:novikov} follows, and this and \cref{eq:2} finish the proof of the lemma.
\end{proof}

Next, we provide estimate on the growth of the process~\(v\).
\begin{lemma}\label{lem:tau-v-relation}
We have
\begin{align}\label{eq:3}
\PP \{ \tau^v < \infty \} &\le \PP \{ \tau^{\hat{v}} < \infty \} + \PP \{ \tau^u < \infty \} + \PP \{  \tau^{u'} < \infty \}, \\\label{eq:4}
 \PP \{  \tau^{\hat{v}} < \infty \} & \le \PP \{  \tau^{\hat{u}'} < \infty \}  +  \lVert\PP -  \Phi^{u,u'}_{\ast} \PP \rVert_{TV}.
\end{align}
In particular, for all~\(\rho \ge 2\) and~\(q\) in \cref{eq:q-range}, 
\begin{equation*}
\PP \{ \tau^v < \infty \},  \PP \{ \tau^{\hat{v}} < \infty \} \le C ( R^{2\bar{p} q}  + 1)\rho^{-q/2 +1} + \frac{1}{2} \Bigl(  \exp ( \lVert u - u' \rVert^2  e^{ C\rho + \epsilon R^{2\bar{p}}} ) - 1 \Bigr)^{1/2}.
\end{equation*}
\end{lemma}

\begin{proof}
On the event~\(\{ \tau^u = \tau^{u'} = \infty \}\), we have~\(\tau^v = \tau^{\hat{v}}\), this gives \cref{eq:3}.

Since~\(\hat{u}' = u'\), \(\hat{v} = v\) before~\(t \le \tau\), by the definition of~\(\Phi^{u,u'}_{\ast}\), we have
\begin{equation*}
\hat{u} \bigl(  t, \Phi^{u,u'}(\omega) \bigr) = u' \bigl( t, \Phi^{u,u'}(\omega) \bigr) = v \bigl( t, \omega\bigr) = \hat{v} (t, \omega), \quad  t \le \tau (\omega).
\end{equation*}
For~\(t \ge \tau\), \(\hat{u}\) and~\(\hat{v}\) follows the same deterministic dynamics, so
\begin{equation}\label{eq:5}
\PP \Bigl\{  \hat{u} (t, \Phi^{u,u'}_{\ast} \omega) = \hat{v}(t, \omega), \quad \forall t \ge 0  \Bigr\} = 1.
\end{equation}
Hence,
\begin{equation*}
\PP \{  \tau^v < \infty \}  = \Phi^{u,u'}_{\ast} \PP \{  \tau^{\hat{u}'} < \infty \} \le \PP \{  \tau^{\hat{u}'} < \infty \} + \lVert \PP -\Phi^{u,u'}_{\ast} \PP \rVert_{TV}.
\end{equation*}
This proves \cref{eq:4}.
\end{proof}

We can further estimate the loss of the maximal coupling. For~\(u, u' \in H\), we define
\begin{equation*}
\Delta (u, u') = d_{TV} \bigl(  \mathcal{L}\{ v(t) \}_{[0, \mathcal{T} ]}, \mathcal{L} \{ u'(t) \}_{[0, \mathcal{T} ]}  \bigr).
\end{equation*}
\begin{lemma}\label{lem:max-coupling-distance}
For every choice of~\(K\) and~\(\rho\), \(\epsilon\), there are constants~\(C_1 = C_1(q)\) and~\(C_2 = C_2(K,\epsilon)\) such that
\begin{equation*}
\Delta (u,u') \le C_1 (R^{2\bar{p}q}+1)  \rho^{-q/2+1} +   \Bigl(  \exp \bigl( d^2  e^{C_2 \rho + \epsilon R^{2\bar{p}}} \bigr) - 1 \Bigr)^{1/2}.
\end{equation*}
\end{lemma}

\begin{proof}
By the definition of the maximal coupling, we have
\begin{align*}
\Delta(u,u') & = \sup_{ \Gamma \in \mathcal{B} \bigl( \mathcal{C}(0, \mathcal{T} ; H) \bigr) } \Big\lvert \PP \{ v(\cdot) \in \Gamma \} - \PP \{ u'(\cdot) \in \Gamma \}\Big\rvert \\
& \le  \PP (\tau < \infty) + \sup_{ \Gamma \in \mathcal{B} \bigl(  \mathcal{C}(0, \mathcal{T} ;H) \bigr)  } \Big\lvert   \PP \{ v(\cdot) \in \Gamma, \ \tau = \infty \} - \PP \{ u'(\cdot) \in \Gamma, \ \tau = \infty \} \Big\rvert \\
& \le \PP (\tau < \infty) +  \sup_{ \Gamma \in \mathcal{B} \bigl(  \mathcal{C}(0, \mathcal{T} ;H) \bigr)  } \Big\lvert   \PP \{ \hat{v}(\cdot) \in \Gamma, \ \tau = \infty \} - \PP \{ \hat{u}'(\cdot) \in \Gamma, \ \tau = \infty \} \Big\rvert \\
& \le \PP (\tau < \infty) + \lVert \PP - \Phi^{u,u'}_{\ast} \PP  \rVert_{TV},
\end{align*}
where the last inequality follows from \cref{eq:5}. The statement then follows from \cref{lem:TV-distance-Phi,lem:tau-v-relation}.
\end{proof}

\section{Proof of Main result}\label{sec:mixing-proof}

The goal of this section is to prove \cref{thm:main-theorem}, by constructing a extension and show that it satisfies the properties \cref{eq:abstract-recurrence,eq:abstract-squeezing} in \cref{thm:main-theorem}. The outline is as follows.
\begin{enumerate}
  \item\label{item:1} Pick~\(\mathcal{T}>0\), and construct a extension as outlined in \cref{sec:constr-mixing-extens}.
  \item\label{item:2} Show that for any ball~\(B = B_H(0,d)\),~\(\tau_{\mathbf{B}}\) has exponential moments. This is done by \cref{lem:exp-recurrence}, which actually specifies the choice of~\(\mathcal{T}\). This verifies \cref{eq:abstract-recurrence}, with~\(g(r) = 1+r^2\).
  \item\label{item:3} Build the stopping time~\(\sigma\) needed for \cref{eq:abstract-squeezing} out of stopping times~\(\tau^u, \tau^{u'}\) introduced in \cref{sec:esti-aux,sec:stopping-time}. This requires choices of the constants~\(K, \rho\). In \cref{sec:poly-squeezing} we will show that if~\(K\) and~\(\rho\) chosen sufficiently large and~\(\delta\) is sufficiently small, then the resulting stopping time will satisfy \cref{eq:abstract-squeezing} with~\(B = B_H(0,\delta)\).
\end{enumerate}
The resulting extension then satisfies the assumptions of \cref{thm:main-theorem}, and \cref{thm:main-theorem} follows immediately from \cref{THACPM}.

\subsection{Recurrence}\label{sec:recurrence}

Recall that we assume~\(N \ge 1\) is arbitrary but fixed, and \cref{eq:span-assumption} holds.
  Let~\(W_h(t) = th + W(t)\). We consider the event where the noise is ``small'' up to time~\(T\):
  \begin{equation*}
    \Omega_{T,\mu} :=  \Bigl\{  \sup_{t \in [0,T] }\lVert W_h(t) \rVert_{1+\kappa} \le \mu\Bigr\}.
  \end{equation*}
  Under the assumption \cref{eq:span-assumption}, the probability~\(\PP(\Omega_{T,\mu}) > 0\) for every~\(T, \delta>0\).
We first prove a result for the solution to equation \cref{eq:conservation-law} under small noise, namely, on \(\Omega_{T,\mu}\).
\begin{lemma}\label{lem:apriori-small-noise}
  Let~\(T > 0\). Then there exists~\(\mu_0 = \mu_0(T,R)\) such that   \begin{equation*}
    \sup_{ t \in [0,T]} \lVert u(t) \rVert^2 \le 2(R^2+1), \quad
    \int_0^T \lVert u(t) \rVert^2_{1+\kappa} \, dt \le C_{a,\nu} T(R^2+1)^{1+\kappa_1}, \qquad \text{ on } \Omega_{T,\mu}, \ \mu \le \mu_0.
  \end{equation*}
\end{lemma}

\begin{proof}
Motivated by the work~\cite{NersesZhao2024PolynomialMixingWhiteforced}, we write \(u= v + W_h\), then~\(v\) is a classical solution to the PDE
  \begin{equation*}
   \partial_t v + F(u)_{x} - \nu D^{2+2\kappa} u + a(v + W_h) = 0.
 \end{equation*}
 Taking inner product with~\(v\) in~\(L^2\), we obtain
 \begin{equation*}
 \frac{d}{dt} \lVert v \rVert^2 + 2a \lVert v \rVert^2 = 2\langle - F(u)_x , v \rangle +2 \nu\langle D^{2+2\kappa} u, v \rangle  - 2a \langle v, W_h \rangle =:  I_1 + I_2 + I_3.
\end{equation*}
Using~\(\langle F(u)_x, u \rangle = 0 \) and integration by parts we have
\begin{equation}\label{eq:super-quad-F}
  \begin{split}
  I_1 &= 2\langle F(u), \partial_x W_h \rangle  \le C\lVert u \rVert^{2+\kappa_1}_{\frac{1}{2}} \mu
        \le  C \lVert u \rVert_{1+\kappa} \lVert u \rVert^{1+\kappa_1} \mu \\
 & \leq\delta \lVert v \rVert^{2(1+\kappa_1)} + C_{\delta} \bigl(  \lVert v \rVert_{1+\kappa}^2\mu^2  + \mu^{2+2\kappa_1} + \mu^4   \bigr).
  \end{split}
\end{equation}
For~\(I_2\) and~\(I_3\) we have
\begin{equation*}
  I_2 \le - 2\nu \lVert D^{1+\kappa} u \rVert^2 + 2\nu \lVert D^{1+\kappa} u \rVert \mu
  \le - \frac{3\nu}{2} \lVert D^{1+\kappa} u \rVert^2 + C_{\nu} \mu^2
\end{equation*}
and
\begin{equation*}
I_3 \le 2a \lVert v \rVert  \lVert W_h \rVert \le \frac{a}{2} \lVert v \rVert^2 + C_a \mu^2.
\end{equation*}
Therefore, we have
\begin{equation*}
  \frac{d}{dt} \lVert v \rVert^2 +\frac{3 (a\wedge \nu)}{2} \lVert v \rVert^2_{1+\kappa}
  \le \delta \lVert v \rVert^{2(1+\kappa_1)} +  C_{\delta} ( \lVert v \rVert_{1+\kappa}^2 \mu^2 + \mu^4 + \mu^2 + \mu^{2+2\kappa_1} ).
\end{equation*}
For every~\(\delta > 0\), by choosing~\(\mu\) sufficiently small, we have
\begin{equation}\label{eq:21}
\frac{d}{dt} \lVert v \rVert^2 + (a\wedge \nu) \lVert v \rVert_{1+\kappa}^2  \le \delta ( \lVert v \rVert^{2(1+\kappa_1)} +  1) \le \delta \bigl( \lVert v \rVert^2+1  \bigr)^{1+\kappa_1} \le \delta \bigl( \lVert v \rVert^2 + 1  \bigr)^3.
\end{equation}
Hence,
\begin{equation*}
\lVert v(t) \rVert^2 \le \frac{\lVert v(0) \rVert^2+1 }{ \sqrt{ 1 - \delta t (\lVert v(0) \rVert^2+1) }}  - 1  \le \sqrt{2}(R^2 +1), \quad \forall t \in [0,T],
\end{equation*}
as long as~\(\delta \le \frac{1}{2T(R^2+1)}\) and then~\(\mu\) sufficiently small so that the first inequality in \cref{eq:21} holds. Using the uniform bound on~\(\lVert v \rVert \), from \cref{eq:21} we can further deduce that
\begin{equation*}
(a\wedge \nu) \int_0^T \lVert v(t) \rVert^2_{1+\kappa} \, dt  \le C T (R^2+1)^{1+\kappa_1}.
\end{equation*}
The conclusion follows since~\(W_h =u-v\) is small.
\end{proof}

\begin{lemma}\label{lem:recurrence}
  Let~\(N \ge 1\) be arbitrary but fixed. For any~\(R, d > 0\), there exists constants~\(p,T > 0\) depending on~\(R, d,a,\nu,h,\mathcal{B}_1,\kappa,\kappa_1\) such that
  \begin{equation*}
  P_T \bigl( u_0 , B_H(0,d) \bigr)   \ge p, \quad \forall u_0 \in B_H(0,R).
\end{equation*}
\end{lemma}

\begin{proof}
  We will show that~\(u \in B_H(0,d)\) on~\(\Omega_{T,\mu}\) for sufficiently large~\(T\) and small~\(\mu\). Let~\(\hat{u}\) be solution to the unforced equation \begin{equation*}
\hat{u}_t + a \hat{u} + \nu (-\partial_{xx})^{1+\kappa} \hat{u} +
\bigl( F(\hat{u}) \bigr)_x =0,
\end{equation*}
it holds that
  \begin{equation}\label{eq:unforced}
   \lVert \hat{u}(t) \rVert \le e^{- at} \lVert u_0 \rVert , \quad \int_0^T \lVert \hat{u}(t) \rVert^2_{1+\kappa} \le \frac{\lVert u_0 \rVert^2 }{a \wedge \nu}.
  \end{equation}
Thus, we can choose~\(T\) large so that~\( \lVert \hat{u}(T) \rVert \le d/3 \).

Let us consider~\(\tilde{w} = u-\hat{u}\) and \(w = \tilde{w} -  W_h \). Then~\(w(0) = 0\) and~\(w(t)\) solves
\begin{equation*}
\partial_t w + F(u)_x - F(\hat{u})_x - \nu D^{2+2\kappa} (w + W_h) + a (w+W_h ) = 0.
\end{equation*}
Taking~\(L^2\)-inner product with~\(w\), we have
\begin{equation*}
  \frac{d}{dt} \lVert w \rVert^2 + 2a \lVert w \rVert^2 + 2\nu \lVert D^{1+\kappa} w \rVert^2 =
  \langle \nu D^{2+2\kappa} W_h - a W_h , w \rangle  + \langle F(u)_x - F(\hat{u})_x, w \rangle =:  J_1 + J_2.
\end{equation*}
We have
\begin{equation*}
J_1 \le \nu \lVert D^{1+\kappa} W_h \rVert \lVert D^{1+\kappa}w \rVert + a \lVert W_h \rVert \lVert w \rVert \le \mu (\nu \lVert D^{1+\kappa} w \rVert + a \lVert w \rVert  ).
\end{equation*}
For~\(J_2\),
\begin{equation*}
  \begin{split}
    J_2 & \le \langle F(u)_x - F( \hat{u} )_x , u - \hat{u} \rangle + \langle F(u) - F(\hat{u}), \partial_x W_h \rangle \\
&          \le \frac{(a\wedge \nu) }{2} \lVert \tilde{w} \rVert_{1+\kappa}^2 + C_{a,\nu } \lVert \tilde{w} \rVert^2 \bigl( e_u^{\kappa_1}(t) + e_{\hat{u}}^{\kappa_1}(t) \bigr) + \langle F(u) - F(\hat{u}), \partial_x W_h \rangle \\
        & \le \frac{(a\wedge \nu) }{2} \lVert w \rVert_{1+\kappa}^2 +
          C_{a,\nu}( \lVert w \rVert^2+\mu) \Bigl( e^{\kappa_1}_u(t) + e^{\kappa_1}_{\hat{u}}(t)   \Bigr),
  \end{split}
\end{equation*}
on~\(\Omega_{\mu}\), where
the second line is essentially repeating the estimate in the proof of \cref{lem:w-naive-growth-bd} with~\(N = 0\) and~\((u,v,w) = (u,\hat{u},\tilde{w})\), and the third line follows from \cref{eq:super-quad-F}.
By Gronwall's inequality, we have
\begin{equation*}
\lVert w(T) \rVert^2 \le \mu e^{ C_{a,\nu} \int_0^T \Phi(t) \, dt}, \quad \Phi(t) = e^{\kappa_1}_u(t) + e^{\kappa_1}_{\hat{u}}(t).
\end{equation*}
By \cref{eq:unforced} and \cref{lem:apriori-small-noise}, for sufficiently small~\(\mu\), we have
\begin{equation*}
\int_0^T \Phi(t) \, dt  \le C_{R,a,\nu,\kappa_1} T, \quad \text{ on } \Omega_{\mu}.
\end{equation*}
Hence, by choosing~\(\mu<d/3\) even smaller, we can have~\(\lVert w(T) \rVert \le d/3 \), and thus~\(\lVert u(T) \rVert \le d \) on~\(\Omega_{\mu}\) as desired.
\end{proof}

Using \cref{lem:growth-estimate-combined-energy} and the estimate \cref{eq:L2-norm-decay-under-coupling}, following the same argument as~\cite[Prop.\ 4.2]{NersesZhao2024ExponentialMixingWhiteForced}, or~\cite[Sect.\ 4.4]{Shirik2008ExponentialMixingRandomly}, one can establish the recurrence of the mixing extension~\((\mathbf{u}_t, \Pp_{\mathbf{u}})\) for an appropriately chosen time step~\(\mathcal{T} = mT\), where~\(T\) is found in \cref{lem:recurrence}.
\begin{lemma}\label{lem:recurrence-for-mixing-extension}
There exists an integer~\(N \ge 1\), depending on \(a, \nu, h, \mathcal{B}_1, \mathcal{B}_{\varphi,m}\), such that for any \(R, d \ge 0\), there exist constants \(p\) and \(\mathcal{T}\) for which the mixing extension constructed in \cref{sec:mixing-extension} using the time step \(\mathcal{T}\) satisfies
  \begin{equation*}
  \Pp_{\mathbf{u}} \bigl(  \mathbf{u}(\mathcal{T}) \in B_H(0,d)^2 \bigr) \ge p, \quad \forall \mathbf{u} \in B_H(0,R)^2,
  \end{equation*}
  provided~\(h\) satisfies \cref{eq:span-assumption}.
\end{lemma}

By \cref{lem:L2-moment-growth},~\(\Phi(u) :=  1 + \lVert u \rVert^2 \) is a Lyapunov function for the family~\((u, \PP_u)\). Hence, by the argument in~\cite[Sect.\ 3]{Shirik2008ExponentialMixingRandomly}, we obtain the recurrence property as follows.
\begin{proposition}\label{lem:exp-recurrence}
  Let~\(N \ge 1\) be as in \cref{lem:recurrence-for-mixing-extension} and~\(h\) satisfy \cref{eq:span-assumption}. For all~\(d > 0\), there exist~\(\delta, \mathcal{T},C\) such that the mixing extension constructed using the time step~\(\mathcal{T}\) satisfies
  \begin{equation*}
  \EE_{\mathbf{u}} e^{ \delta \tau_d} \le C \bigl(  1 + \lVert u \rVert^2 + \lVert u' \rVert^2   \bigr).
  \end{equation*}
\end{proposition}

\subsection{Polynomial squeezing}\label{sec:poly-squeezing}

Let~\(\mathcal{T}\) be chosen at the end of \cref{sec:recurrence}, and suppose that we have constructed the extension using~\(\mathcal{T}\) as outlined in \cref{sec:constr-mixing-extens}.
For the process in the extension~\((\tilde{u}, \tilde{u}')\), we define~\(\sigma:=  \tilde{\tau} \wedge \sigma_1\) where
\begin{equation*}
 \tilde{\tau} :=  \tau^{\tilde{u}} \wedge \tau^{ \tilde{u}'}, \quad \sigma_1 := \inf_{ t \ge 0 }  \{  \tilde{u}'(t) \neq \tilde{v}(t) \},
\end{equation*}
where~\(\tau^{\tilde{u}}, \tau^{\tilde{u}'}\) are introduced in \cref{sec:stopping-time,sec:esti-aux}, and both depend on constants~\(K \) and~\(\rho\).
\begin{lemma}\label{lem:probability-of-Q}
Let~\(q\) be in \cref{eq:q-range}. There exist sufficiently large~\(K, \rho > 0\) and sufficiently small~\(\delta > 0\) such that for all~\(\mathbf{u} \in B_H(0,\delta)^2\),
\begin{equation*}
\Pp_{\mathbf{u}} (\sigma \in [k \mathcal{T}, (k+1)\mathcal{T}]) \le \frac{1}{  2 (k+1)^{q/2}}, \quad k \ge 0.
\end{equation*}
\end{lemma}

\begin{proof}
Let~\(\mathcal{Q}_k  = \{  \sigma \in [k \mathcal{T}, (k+1) \mathcal{T}] \}\), \(k \ge 0\). Conditioning on~\(\Fc_{k \mathcal{T}}\), by the Markov property of the dynamics at~\(\mathcal{T} \Z\), we have
\begin{equation*}
\begin{split}
\Pp_{\mathbf{u}} ( \mathcal{Q}_k ) &= \Ee_{\mathbf{u}} \Bigl(  \ONE_{\{  \sigma \ge k \mathcal{T} \}} \Ee [ \mathcal{Q}_k | \Fc_{k \mathcal{T}}]  \Bigr) = \Ee_{\mathbf{u}} \bigl(  \ONE_{\{ \sigma \ge k \mathcal{T} \}} \Ee [  \mathcal{Q}_k | \mathbf{u}(k \mathcal{T}) ] \bigr)  \\
& \le \Pp_{\mathbf{u}} (k \mathcal{T} \le \tilde{\tau} < \infty ) + \Ee_{\mathbf{u}} \ONE_{\{  \sigma \ge k \mathcal{T} \}} \Delta \bigl(  u(k \mathcal{T}), u'(k \mathcal{T}) \bigr) =:  P_1 + P_2.
\end{split}
\end{equation*}
By \cref{lem:growth-estimate-combined-energy}, the first term~\(P_1\) is bounded by~\(C_{q,K}(\delta^{2 \bar{p}q}+1)  (k \mathcal{T}+2)^{  -q/2} \le \frac{1}{6   (k+1)^{q/2} }\) if~\(\delta\) is sufficiently small and~\(\rho\) sufficiently large. For~\(P_2\), by \cref{lem:max-coupling-distance} with~\(\rho\) replaced by~\(\rho + \mu  k \mathcal{T}\) with~\(\mu >0 \) to be specified later,
\begin{equation*}
P_2 \le \Ee_{\mathbf{u}} \ONE_{\{  \sigma \ge k \mathcal{T} \}} \cdot \Bigl[C ( 2R^{2\bar{p}q}+ 1) (\rho + \mu k \mathcal{T})^{  - q/2} +  \Bigl(   \exp \bigl(  d_k^2  e^{ C_{K,\epsilon}( \rho + \mu k \mathcal{T}) + \epsilon R^{2\bar{p}}_{k}} \bigr) - 1 \Bigr)^{1/2}   \Bigr] :=  P_{21} + P_{22},
\end{equation*}
where~\(R_k = \lVert \tilde{u}(k \mathcal{T}) \rVert \vee  \lVert \tilde{u}'(k \mathcal{T}) \rVert  \) and~\(d_k = \lVert \tilde{u}(k \mathcal{T}) - \tilde{u}'(k \mathcal{T}) \rVert \).
For~\(P_{21}\), we have
\begin{equation*}
P_{21} \le C \frac{\Ee_{\mathbf{u}} \lVert \tilde{u}(k \mathcal{T}) \rVert^{2\bar{p}q}+ \lVert \tilde{u}'(k \mathcal{T}) \rVert^{2\bar{p}q} + 1  }{    (\rho + \mu k \mathcal{T})^{q/2} } \le C_q \frac{ e^{ - a \bar{p} q k \mathcal{T} } 2R^{2\bar{p}q} + 1}{   (\rho + \mu k \mathcal{T})^{q/2}  }  \le \frac{1}{  6(k+1)^{q/2} }
\end{equation*}
using \cref{lem:L2-moment-growth}.
Noting that $ \tilde{v}(t)  = \tilde{u}'(t) $ for $t \le \sigma$, by \cref{eq:L2-norm-decay-under-coupling}, for some constant~\(C'_{K,\epsilon} > 0\) we have
\begin{equation}\label{eq:16}
d_k^2 \le d^2 \exp \Bigl(  - \frac{a}{2}k \mathcal{T} +  C'_{K,\epsilon}( \rho + \mu k \mathcal{T})+ \epsilon R^{2\bar{p}}/2  \Bigr),
\end{equation}
and by the definition of~\(\sigma\),
\begin{equation}\label{eq:17}
  R^{2\bar{p}}_k  \le (K+1) k \mathcal{T} + 2K R^{2\bar{p}} + 2 \rho.
\end{equation}
Using \cref{eq:16,eq:17}, we have~\(P_{22} \le (\exp (d^2 e^{\mathcal{A}}) -1)^{1/2}\) where
\begin{equation*}
  \begin{split}
    \mathcal{A} &\le
        \Bigl\{    - \frac{a}{2}k \mathcal{T} +  C_{K,\epsilon}'( \rho + \mu k \mathcal{T})+ \epsilon R^{2\bar{p}}/2 \Bigr\}
        + \Bigl\{ C_{K,\epsilon} (\rho + \mu  k \mathcal{T}) + \epsilon \bigl(  (K+1) k \mathcal{T} + 2K R^{2\bar{p}} + 2 \rho \bigr)  \Bigr\} \\
      &
        =:  - \frac{a}{2} k \mathcal{T} + \mathcal{A}_1 k \mathcal{T} + \mathcal{A}_2 \rho + \mathcal{A}_3 R^{2 \bar{p}}.
  \end{split}
\end{equation*}
where~\(\mathcal{A}_i\) are expressions in~\(C_{K,\epsilon}, C_{K,\epsilon}', K, \epsilon\) and~\(\mu \). In particular,
\begin{equation*}
\mathcal{A}_1 = \mu  ( C_{K,\epsilon} + C_{K,\epsilon}' ) + \epsilon (K + 1).
\end{equation*}
By first choosing~\(\epsilon\) small, and then~\(\mu \) sufficiently small, we can guarantee that~\(\mathcal{A}_1 \le a/4\), and hence
\begin{equation*}
  P_{22} \le \Bigl[  \exp \Bigl(  d^2 e^{ - \frac{a}{4} k \mathcal{T} + \mathcal{C}(\rho + R^{2 \bar{p}}  )} \Bigr)  - 1\Bigr]^{1/2}
  \le \Bigl[  \exp \Bigl(  4 \delta^2 e^{ - \frac{a}{4} k \mathcal{T} + \mathcal{C}(\rho + \delta^{2\bar{p}})}  \Bigr) - 1 \Bigr]^{1/2}.
\end{equation*}
By choosing~\(\delta \) sufficiently small, we have~\(P_{22} \le \frac{1}{  6(k+1)^{q/2}} \), and this completes the proof.
\end{proof}

\bigskip
Finally, let us verify \cref{eq:abstract-squeezing} with~\(B = B_H(0,\delta)\) and~\(\sigma\) constructed using sufficiently large~\(K\) and~\(\rho\).

\begin{proof}[Derivation of \cref{eq:abstract-squeezing} from \cref{lem:probability-of-Q}]

  Note that~\(g(r) = 1+r^2\). Let us verify all the inequalities in \cref{eq:abstract-squeezing} for any~\(p_0 \in   (1, q^*/2-1) \) with~\(q^*> 4\).
Indeed, for any fixed~\(p_0 \in    (1, q^*/2-1) \), there exists some \(q\in   (4,q^*) \) such that \(p_0 \in   (1, q/2-1) \).

By \cref{eq:16,eq:17}, for~\(t \le \sigma\),~\(\lVert \tilde{u}(t) - \tilde{u}'(t) \rVert  \) decays exponentially, so the first inequality in \cref{eq:abstract-squeezing} holds for all~\(p_0\). By \cref{lem:probability-of-Q}, we have
  \begin{equation*}
  \Pp_{\uu} ( \sigma = \infty) \ge 1 - \sum_{k = 0}^{\infty} \Pp_{\uu} \Bigl\{  \sigma \in [k \mathcal{T}, (k+1)\mathcal{T}] \Bigr\} \ge 1 -\frac{1}{2} \sum_{k = 0}^{\infty} \frac{1}{   (1+k)^{q/2 }} \ge 1 - \frac{\pi^2}{12} >  0.
\end{equation*}
This is the second inequality in \cref{eq:abstract-squeezing}. The third inequality follows from
\begin{equation*}
  \begin{split}
    \Ee_{\uu} ( \ONE_{ \{  \sigma < \infty \} } \sigma^{p_0} )
    &\le \sum_{k = 0}^{\infty} \Ee_{\uu} \Bigl( \ONE_{\{  \sigma \in [k\mathcal{T}, (k+1) \mathcal{T}] \}} \sigma^{p_0} \Bigr) \\
    &\lesssim_{q,\mathcal{T}} \sum_{k = 0}^{\infty} (k+1)^{p_0} \cdot \frac{1}{  2(1+k)^{q/2 }} \lesssim_{q, \mathcal{T}} \sum_{k = 0}^{\infty} \frac{1}{   (k+1)^{ q/2 - p_0 }  }< \infty.
  \end{split}
\end{equation*}
The fourth inequality follows from
\begin{equation*}
\Ee_{\uu} \Bigl(  \ONE_{\{  \sigma < \infty \}} \bigl(  \lVert \tilde{u}(\sigma) \rVert^{2p_0} + \lVert \tilde{u}'(\sigma) \rVert^{2 p_0}   \bigr) \Bigr) \le C_{p_0 } \bigl(  \lVert \tilde{u}(0) \rVert^{2p_0} + \lVert \tilde{u}'(0) \rVert^{2p_0} +   \Ee_{\uu }  \bigl(  \ONE_{ \{ \sigma < \infty \}} ( 1 + \sigma^{p_0} ) \bigr)   \bigr) \le C_{p_0}.
\end{equation*}
This finishes the proof.
\end{proof}

\appendix
\section{Auxiliary results}\label{sec:appendix}
\subsection{Functional inequalities}\label{sec:funct-ineq}
Commutator estimate.
\begin{lemma}\label{lem:commutator}
For~\(\kappa > 0\),
\begin{equation}\label{eq:commutator-s-ge-1-endpt}
\lVert D^{1+\kappa} (fg) - f D^{1+\kappa} g \rVert_{L^2} \lesssim \lVert \partial_x f \rVert_{L^{\infty}} \lVert D^{\kappa} g \rVert_{L^2} + \lVert D^{1+\kappa} f \rVert_{L^{\infty}} \lVert g \rVert_{L^2}.
\end{equation}
\end{lemma}

\begin{proof}
By the fractional Leibniz rule in~\cite[Theorem 1.2 (1.6)]{Li2019KatoPonceFractional} with~\(p = p_2 = 2\), \(p_1 = \infty\), \(s_1 = 1\) and~\(s_2 = \kappa\), we have
\begin{equation*}
 \lVert D^{1+\kappa} (fg) - f D^{1+\kappa} g - \partial_x f \tilde{D}^{\kappa} g - g D^{1+\kappa} f  \rVert_{L^2} \lesssim \lVert D f \rVert_{BMO} \lVert D^{\kappa} g \rVert_{L^2},
\end{equation*}
where
\begin{equation*}
(\tilde{D}^{\kappa} h)^{\wedge} (\xi) :=  i^{-1} (1+\kappa) \lvert \xi \rvert^{\kappa} \sgn(\xi) \cdot \hat{h}(\xi).
\end{equation*}
Note that~\(\lVert \tilde{D}^{\kappa} h \rVert_{L^2} = (1+\kappa) \lVert D^{\kappa} h \rVert_{L^2}\) by Plancherel equality.

We have
\begin{equation*}
\lVert \partial_x f \cdot \tilde{D}^{\kappa} g \rVert_{L^2} \le \lVert \partial_x f \rVert_{L^{\infty}} \lVert  \tilde{D}^{\kappa} g \rVert_{L^2}     = (1+\kappa) \lVert \partial_x  f\rVert_{L^{\infty}} \lVert D^{\kappa} g \rVert_{L^2},
\end{equation*}
and by the boundedness of Hilbert transform~\(\mathcal{H}: L^{\infty} \to BMO\),
\begin{equation*}
 \lVert D f \rVert_{BMO} = \lVert \mathcal{H} (\partial_x f)  \rVert_{BMO} \lesssim  \lVert \partial_x f \rVert_{L^{\infty}}.
\end{equation*}
The conclusion follows.
\end{proof}

Kato--Ponce inequality (or fractional Leibniz rule)
\begin{lemma}{\cite[Theorem 1]{GrafakOh2014KatoPonceInequality}}\label{lem:kato-ponce}
For~\(s \in (0,1)\) and~\(1/p = 1/p_1 + 1/q_1 = 1/p_2 +1/q_2\), \(p,p_1,p_2,q_1, q_2 \in  (1,\infty\Bigr]\),
\begin{equation*}
\lVert D^s (fg) \rVert_{L^p} \lesssim \lVert D^s f \rVert_{L^{p_1}} \lVert g \rVert_{L^{q_1}} +  \lVert f \rVert_{L^{p_2}} \lVert D^{s}g \rVert_{L^{q_2}}.
\end{equation*}
\end{lemma}

As a corollary, we have the following.
\begin{lemma}\label{lem:fractional-product}
for~\(s \in (0,1)\), we have
\begin{equation}\label{eq:Hs-norm-of-product}
\lVert fg \rVert_s \lesssim \lVert f \rVert_s \bigl(  \lVert g \rVert_{L^{\infty}} + \lVert D^s g \rVert_{L^{\infty}}   \bigr) \lesssim \lVert f \rVert_s \lVert g \rVert_{\frac12 + 2s} .
\end{equation}
\end{lemma}

\begin{proof}
The first inequality follows from \cref{lem:kato-ponce} with~\(p = p_1 = p_2 = 2\) and~\(q_1 = q_2 = \infty\). Since~\(H^{\frac12 +s } (\R) \subset L^{\infty}(\R)\), we have
\begin{equation*}
\lVert g \rVert_{L^{\infty}} \lesssim  \lVert g \rVert_{\frac12 + 2s}, \quad \lVert D^s g \rVert_{L^{\infty}} \lesssim \lVert D^s g \rVert_{\frac12 + s} =   \lVert (1+ \lvert \xi \rvert^2 )^{\frac{1+2s}{4}} \lvert \xi \rvert^s  \hat{g} \rVert \le \lVert (1+ \lvert \xi \rvert^2 )^{\frac{1+4s}{4}} \hat{g} \rVert = \lVert g \rVert_{\frac12 + 2s},
\end{equation*}
and the second inequality follows.
\end{proof}

\begin{lemma}\label{lem:Lip-sobolev-embedding}
Let~\(\Phi: \R \to \R\) to be a global Lipschitz function and~\(s \in (0,1)\). Then
\begin{equation*}
\lVert \Phi(u) \rVert_{s} \lesssim \lVert \Phi \rVert_{\mathrm{Lip}} \lVert u \rVert_{s}.
\end{equation*}
\end{lemma}

\begin{proof}
Recall that for~\(s \in (0,1)\), the Slobodeckij seminorm
\begin{equation*}
[ u ]_{H^s(\R^n)} =   \Bigl( \iint_{\mathbb{R}^n \times \mathbb{R}^n} \frac{|u(x) - u(y)|^2}{|x-y|^{n+2s}} \,dx\,dy \Bigr)^{1/2}
\end{equation*}
is equivalent to the~\(H^s\)-seminorm~\(\lVert \cdot \rVert_{\dot{H}^s}\). Let~\(L = \lVert \Phi \rVert_{\mathrm{Lip}}\). We have
\begin{equation*}
\lVert \Phi(u) \rVert_{\dot{H}^s} \lesssim [ \Phi(u)]_{H^s}  \le  L [u]_{H^s} \lesssim L \lVert u \rVert_{\dot{H}^s}.
\end{equation*}
The conclusion follows.
\end{proof}

The following deals with~\(H^{1+s}\), which is a corollary of the general result on composition operator norm, see\cite[Theorem 2, Chap 5.3.6]{RunstSickel2011SobolevSpacesFractional}.
\begin{lemma}\label{lem:composition-operator}
  Let \(\Phi \in \mathrm{Lip}_{1+\alpha}  \) with \(\alpha > \kappa' \). Then there exists \(C = C(\Phi,\kappa') \) such that
\begin{equation*}
\lVert \Phi (u) \rVert_{1+\kappa'} \le C \lVert u \rVert^{\alpha}_{L^{\infty}} \lVert u \rVert_{1+\kappa'}.
\end{equation*}
\end{lemma}

\medskip

Finally, for convenience we record the following interpolation inequality for Sobolev spaces.
\begin{lemma}\label{lem:sobolev-interpolation}
  Let~\(0 \le s_0 < s_1 \leq s_2 < s_3\) and~\(\alpha_1 + \alpha_2  = \alpha_0 + \alpha_3\).
  Then
  \begin{equation*}
  \lVert f \rVert_{s_1}^{\alpha_1} \lVert f \rVert_{s_2}^{\alpha_2}  \le \lVert f \rVert^{\alpha_0}_{s_0} \lVert f \rVert^{\alpha_3}_{s_3},
  \end{equation*}
  provided that~\(\alpha_1s_1 + \alpha_2 s_2  \le \alpha_0 s_0 + \alpha_3 s_3 \).
\end{lemma}
\begin{proof}
  Recall that the Sobolev norm~\(\lVert f \rVert_{H^s} := \lVert  ( 1 + \lvert \xi \rvert^2  )^{s/2} \hat{f} \rVert_{L^2}  \) is log-convex and increasing in~\(s\). Let~\(t_{0}\) satisfy~\(\alpha_1 s_1 + \alpha_2 s_2 = \alpha_0 t_0 + \alpha_3 s_3 \). Then~\(t_0 \le s_0\) by assumption. By log-convexity and monotonicity, we have
  \begin{equation*}
  \lVert f \rVert_{s_1}^{\alpha_1} \lVert f \rVert_{s_2}^{\alpha_2}    \le \lVert f \rVert^{\alpha_0}_{t_0} \lVert f \rVert^{\alpha_3}_{s_3} \le    \lVert f \rVert^{\alpha_0}_{s_0} \lVert f \rVert^{\alpha_3}_{s_3}.
\end{equation*}
This completes the proof.
\end{proof}

\subsection{Fractional derivatives of space-time weight function \texorpdfstring{\(\psi\)}{ψ}}

\begin{lemma}\label{lem:frac-deri-of-psi}
  Let~\(\psi(t,x)\) be defined in \cref{eq:def-psi}. Then for all~\(m \ge 1\),
  \begin{equation*}
    \sup_{t \ge 0} \lVert D^s \psi^m(t, \cdot) \rVert_{L^{\infty}} \le C(s,m), \ s \ge 1, \quad
    \sup_{t \ge 0} \lVert \partial_x^k \psi^m(t, \cdot) \rVert_{L^{\infty}} \le C(k,m), \ k \ge 1.
  \end{equation*}
\end{lemma}

\begin{proof} Let~\(\phi = \psi^m\). We will establish
\begin{equation}\label{eq:23}
\sup_{t \ge 0} \lvert \partial_x^k \phi(t,x) \rvert \lesssim ( 1 + \lvert x \rvert  )^{- k + \eps}, \quad k \ge  1, \ \forall \eps > 0,
\end{equation}
which will imply the first estimate as well.
Indeed, let~\(g_t (x) = \partial_x \phi(t,x)\). Then
\begin{equation*}
\bigl(  D^s \phi(t, \cdot)   \bigr)^{\wedge} = -i \sgn(\xi) \lvert \xi \rvert^{s-1} \hat{g}_t(\xi),
\end{equation*}
and the statement of the lemma follows from
\begin{equation*}\label{eq:13}
\sup_{t \ge 0} \int_{\R} \lvert \xi \rvert^{s-1} \lvert  \hat{g}_t(\xi) \rvert \, d \xi < \infty.
\end{equation*}
By \cref{eq:23},
\begin{equation*}
\sup_{t \ge 0 } \lVert g_t \rVert_{L^2} < \infty, \quad \sup_{ t \ge 0 } \lVert \partial_x^k g_t \rVert_{L^1} < \infty, \quad k \ge 1.
\end{equation*}
Choosing~\(\N \ni k_0 > s\), we have
\begin{equation*}
\lvert \hat{g_t}(\xi) \rvert = \lvert \xi \rvert^{-k} \widehat{ \partial_x^k g_t }(\xi) \le \frac{C_k}{\lvert \xi \rvert^k }.
\end{equation*}
Hence,
\begin{align*}
  \int_{\R} \lvert \xi \rvert^{s-1} \hat{g}_t(\xi) \, d \xi
  & \le \Bigl[ \int_{\lvert \xi \rvert \le 1 } \lvert \xi \rvert^{2s-2} \, d \xi  \Bigr]^{1/2} \lVert g_t \rVert_{L^2} +  \int_{\lvert \xi \rvert \ge 1 } C_k \lvert \xi \rvert^{-k + s -1} \, d \xi < \infty.
\end{align*}

It remains to show \cref{eq:23}. We will sketch the computation by using the Fa\`adi Bruno's formula:
\begin{equation*}\label{eq:FB-chain}
\frac{d}{d x^k} f\bigl(g(x)\bigr) = \sum_{\substack{ \alpha_j \ge 0, \\ \alpha_1 + 2\alpha_2 + \dotsb + k\alpha_k = k } } \frac{k!}{\alpha_1! \dotsb \alpha_k!} f^{(\alpha_1 + \dotsb + \alpha_k)}\bigl(g(x)\bigr) \prod_{ \ell=1}^k \Bigl( \frac{ g^{(\alpha_j)}(x)}{ \ell!}\Bigr)^{\alpha_{\ell}}.
\end{equation*}
We write~\( \psi(t,x) = \varphi(x) \cdot  h(u) \), where~\(u = t/\varphi(x)\). Then~\(h(u) = 1- e^{-u} \in [0,1]\). For~\(k \ge 1\), we have
\begin{equation}\label{eq:24}
 \lvert\partial_x^k \phi(t,x)   \rvert \le \varphi(x)^m  \sum_{ \alpha_1 + \dotsb + k \alpha_k = k} C_{\alpha,k}\prod_{ \ell = 1 }^k \Big\lvert   \frac{\partial_x^{\ell} \psi(t,x)}{ h(u) }  \Big\rvert^{\alpha_{\ell}}.
\end{equation}
For every~\(\ell \ge 1\), we have
\begin{align*}
  \partial_x^{\ell} \psi(t,x) &= \sum_{ s = 0}^{\ell} \varphi^{(\ell-s)}(x) \sum_{ \beta_1 + \dotsb + s \beta_s = s }C_{\beta} h^{(\beta_1 + \dotsb + \beta_s)} ( u )  \prod_{ p = 1}^{^s} \bigl[ \partial_x^p \Bigl(  \frac{t}{\varphi(x)} \Bigr) \bigr]^{\beta_p} \\
                 &= \sum_{ s = 0} \varphi^{(\ell-s)}(x) \sum_{ \beta_1 + \dotsb + s \beta_s = s }C_{\beta} h^{(\beta_1 + \dotsb + \beta_s)} (  u )  \prod_{ p = 1}^{^s} \Bigl[ u \sum_{ \gamma_1 + \dotsb +p  \gamma_p = p} C_{\gamma} \frac{ \prod_{q=1}^p \varphi^{(q)}(x)}{\varphi^{\gamma_1 + \dotsb + \gamma_p}}  \Bigr]^{\beta_p} \\
  & = \varphi^{(\ell)}(x) h(u) +  \sum_{ s = 1}^{\ell} \varphi^{(\ell-s)}(x) \sum_{ \substack{ a \ge 1, \ b \ge 1 \\  \gamma_1 + \dotsb + s \gamma_s = s} }C_{a,b,\gamma} h^{(a)} ( u) u^a  \varphi^{-b} \prod_{q=1}^s \bigl[   \varphi^{(q)}(x) \bigr]^{\gamma_q},
\end{align*}
Since~\(\frac{u}{ h(u)}\),~\(h^{(a)}(u), a \ge 1\) are bounded, and~\(\lvert \varphi^{(q)}(x) \rvert \lesssim (1+\lvert x \rvert )^{-q} \),~\(q \ge 1\), we have
\begin{equation*}
\Big\lvert \frac{\partial_x^{\ell} \psi(t,x)}{ h \bigl(  t/\varphi(x) \bigr)}  \Big\rvert \lesssim_{\ell} (1+ \lvert x \rvert )^{ - \ell}.
\end{equation*}
Plugging this into \cref{eq:24}, we obtain
\begin{equation*}
\lvert \partial_x^k \phi(t,x) \rvert \lesssim (1+ \lvert x \rvert )^{-k} \varphi(x)^m \lesssim ( 1 + \lvert x \rvert  )^{-k+\eps},
\end{equation*}
as desired.

\end{proof}

\subsection{Martingale inequalities}\label{sec:mart-ineq}

We recall some martingale inequalities.
\begin{lemma}\label{lem:super-mart-ineq}
Let~\(X_t\) be a non-negative super-martingale, then for every~\(x > 0\),
\begin{equation*}
\PP( \sup_{t \ge 0 } X_{t } \ge x  )  \le \frac{\E X_0}{ x}.
\end{equation*}
In particular, if~\(M_t\) is a local martingale, then
\begin{equation*}
\PP ( \sup_{t \ge 0 } (M_t - \gamma \langle M \rangle_t ) \ge \rho  )  = \PP ( \sup_{ t \ge 0 } e^{ \gamma M_t - \frac{\gamma^2}{2} \langle M \rangle_t } \ge e^{  \gamma \rho}  )\le e^{ - \gamma \rho}.
\end{equation*}
\end{lemma}

\begin{lemma}\label{lem:martingale-linear-growth}
Let~\(M_t\) be a continuous martingale with~\(M_0 = 0\). Assume that for some~\(q >0\),
\begin{equation}\label{eq:11}
\EE \langle M \rangle_t^{q/2} \le K_q  t^{q/2}, \quad t \ge 1.
\end{equation}
Then for every~\(\gamma > 0\), there exists a constant~\(C = C_{q,\gamma}\) such that
\begin{equation*}
\PP \bigl( \sup_{t \ge 0 }  (M_t - \gamma t) \ge \rho   \bigr) \le C_{q,\gamma} K_q\rho^{ - \frac{q}{2} }, \quad \rho \ge 2.
\end{equation*}
\end{lemma}

\begin{proof}
  Let~\(k_0 \in \N\) be such that~\(2^{k_0} \le \rho < 2^{k_0+1}\).
  We have
  \begin{equation*}
  \PP \bigl(  \sup_{t \ge 0} (M_t - \gamma t) \ge \rho \bigr) \le \PP ( \sup_{ 0 \le t \le 2^{k_0}} (M_t - \gamma t ) \ge \rho  ) + \sum_{k = k_0}^{\infty} \PP \bigl( \sup_{t \in [2^k, 2^{k+1}]} (M_t - \gamma t ) \ge \rho \bigr) =: p_1 + p_2.
\end{equation*}
By Doob's maximal inequality, Burkholder--Davis--Gundy inequality and \cref{eq:11}, we have
\begin{equation}\label{eq:32}
p_1 \le \PP \bigl(  \sup_{0 \le t \le 2^{k_0} } M_t \ge \rho \bigr) \le \frac{ \EE M_{2^{k_0}}^q}{ \rho^q} \le \frac{ K_q 2^{k_0 q/2}}{ \rho^q} \le K_q \rho^{-q/2}.
\end{equation}
For~\(p_2\), similar to \cref{eq:32} we have
\begin{equation*}
\PP \bigl(  \sup_{t \in [2^k, 2^{k+1}]} (M_t - \gamma t) \ge \rho \bigr) \le \PP \bigl( \sup_{t \in [2^k, 2^{k+1}] } M_t \ge \rho + \gamma 2^k  \bigr) \le  \PP \bigl(  \sup_{t \in [0, 2^{k+1}] } M_t \ge \gamma 2^k  \bigr) \le K_q \frac{2^{(k+1)q/2}}{ ( \gamma 2^k )^{q}}, 
\end{equation*}
and hence
\begin{equation*}
p_2  \lesssim_{q,\gamma} K_q \sum_{ k = k_0}^{\infty}   2^{ - k q/2} \lesssim_{q,\gamma} K_q 2^{ - k_0 q /2} \lesssim_{q,\gamma} K_q \rho^{ - q/2}. 
\end{equation*}
The conclusion follows.
\end{proof}

\bibliographystyle{alpha}

\begin{thebibliography}{DDG{\etalchar{+}}23}

\bibitem[Abe89]{abergel1989attractor}
F~Abergel.
\newblock Attractor for a {{Navier}}--{{Stokes}} flow in an unbounded domain.
\newblock {\em ESAIM: Mod{\'e}lisation math{\'e}matique et analyse
  num{\'e}rique}, 23(3):359--370, 1989.

\bibitem[ACQ11]{AmirCorwinQuaste2011ProbabilityDistributionFree}
Gideon Amir, Ivan Corwin, and Jeremy Quastel.
\newblock Probability distribution of the free energy of the continuum directed
  random polymer in 1 + 1 dimensions.
\newblock {\em Communications on Pure and Applied Mathematics}, 64(4):466--537,
  2011.

\bibitem[Ali07]{Alibau2007EntropyFormulationFractal}
Nathaël Alibaud.
\newblock Entropy formulation for fractal conservation laws.
\newblock {\em Journal of Evolution Equations}, 7(1):145--175, 2007.

\bibitem[AV24]{AgrestVeraar2024CriticalVariationalSetting}
Antonio Agresti and Mark Veraar.
\newblock The critical variational setting for stochastic evolution equations.
\newblock {\em Probability Theory and Related Fields}, 188(3-4):957--1015,
  April 2024.

\bibitem[Bak13]{Bakhti2013BurgersEquationPoisson}
Yuri Bakhtin.
\newblock The {{Burgers}} equation with {{Poisson}} random forcing.
\newblock {\em Annals of Probability}, 41(4):2961--2989, July 2013.

\bibitem[Bak16]{Bakhti2016InviscidBurgersEquation}
Yuri Bakhtin.
\newblock Inviscid {{Burgers}} equation with random kick forcing in noncompact
  setting.
\newblock {\em Electronic Journal of Probability}, 21:50 pp., 2016.

\bibitem[Ban19]{Banerj2019FractionalHyperviscosityInduced}
Debarghya Banerjee.
\newblock Fractional hyperviscosity induced growth of bottlenecks in energy
  spectrum of {{Burgers}} equation solutions.
\newblock {\em The European Physical Journal B}, 92(9):209, 2019.

\bibitem[BCK14]{bakhtin2014space}
Yuri Bakhtin, Eric Cator, and Konstantin Khanin.
\newblock Space-time stationary solutions for the {{Burgers}} equation.
\newblock {\em Journal of the American Mathematical Society}, 27(1):193--238,
  2014.

\bibitem[BK10]{BakhtiKhanin2010LocalizationPerronFrobeniusTheory}
Yuri Bakhtin and Konstantin Khanin.
\newblock Localization and {{Perron-Frobenius}} theory for directed polymers.
\newblock {\em Moscow Mathematical Journal}, 10(4):667--686, 838, 2010.

\bibitem[BK13]{BoritcKhanin2013HyperbolicityMinimizers1D}
Alexandre Boritchev and Konstantin Khanin.
\newblock On the hyperbolicity of minimizers for {{1D}} random {{Lagrangian}}
  systems.
\newblock {\em Nonlinearity}, 26(1):65--80, January 2013.

\bibitem[BK18]{BakhtiKhanin2018GlobalSolutionsRandom}
Yuri Bakhtin and Konstantin Khanin.
\newblock On global solutions of the random {{Hamilton}}--{{Jacobi}} equations
  and the {{KPZ}} problem.
\newblock {\em Nonlinearity}, 31(4):R93--R121, February 2018.

\bibitem[BK21]{BoritcKuksin2021OnedimensionalTurbulenceStochastic}
Alexandre Boritchev and Sergei Kuksin.
\newblock {\em One-Dimensional Turbulence and the Stochastic {{Burgers}}
  Equation}, volume 255 of {\em Mathematical {{Surveys}} and {{Monographs}}}.
\newblock American Mathematical Society, Providence, Rhode Island, July 2021.

\bibitem[BL18]{BakhtiLi2018ZeroTemperatureLimit}
Yuri Bakhtin and Liying Li.
\newblock Zero temperature limit for directed polymers and inviscid limit for
  stationary solutions of stochastic {{Burgers}} equation.
\newblock {\em Journal of Statistical Physics}, 172(5):1358--1397, September
  2018.

\bibitem[BL19]{bakhtin2019thermodynamic}
Yuri Bakhtin and Liying Li.
\newblock Thermodynamic limit for directed polymers and stationary solutions of
  the {{Burgers}} equation.
\newblock {\em Communications on Pure and Applied Mathematics}, 72(3):536--619,
  2019.

\bibitem[Bor18]{Boritc2018ExponentialConvergenceStationary}
Alexandre Boritchev.
\newblock Exponential convergence to the stationary measure for a class of {{1D
  Lagrangian}} systems with random forcing.
\newblock {\em Stochastics and Partial Differential Equations: Analysis and
  Computations}, 6(1):109--123, March 2018.

\bibitem[Bur40]{Burger1940HydrodynamicsApplicationModel}
J.~M. Burgers.
\newblock Hydrodynamics. — application of a model system to illustrate some
  points of the statistical theory of free turbulence.
\newblock {\em Nederl. Akad. Wetensch., Proc.}, 43:2--12, 1940.

\bibitem[Bur74]{Burger1974NonlinearDiffusionEquation}
J.M. Burgers.
\newblock {\em The Nonlinear Diffusion Equation: Asymptotic Solutions and
  Statistical Problems}.
\newblock D. Reidel Pub. Co, 1974.

\bibitem[CP19]{chen2019invariant}
Gui-Qiang~G Chen and Peter~HC Pang.
\newblock Invariant measures for nonlinear conservation laws driven by
  stochastic forcing.
\newblock {\em Chinese Annals of Mathematics, Series B}, 40(6):967--1004, 2019.

\bibitem[DDG{\etalchar{+}}23]{drivas2023invariant}
Theodore~D Drivas, Alexander Dunlap, Cole Graham, Joonhyun La, and Lenya
  Ryzhik.
\newblock Invariant measures for stochastic conservation laws on the line.
\newblock {\em Nonlinearity}, 36(9):4553--4584, 2023.

\bibitem[DGR21]{dunlap2021stationary}
Alexander Dunlap, Cole Graham, and Lenya Ryzhik.
\newblock Stationary solutions to the stochastic {{Burgers}} equation on the
  line.
\newblock {\em Communications in Mathematical Physics}, 382(2):875--949, 2021.

\bibitem[Dro03]{Dronio2003VanishingNonlocalRegularization}
Jerome Droniou.
\newblock Vanishing non-local regularization of a scalar conservation law.
\newblock {\em Electronic Journal of Differential Equations}, 2003(117):1--20,
  2003.

\bibitem[DS05]{DirrSougan2005LargetimeBehaviorViscous}
Nicolas Dirr and Panagiotis~E. Souganidis.
\newblock Large-time behavior for viscous and nonviscous {{Hamilton-Jacobi}}
  equations forced by additive noise.
\newblock {\em SIAM Journal on Mathematical Analysis}, 37(3):777--796
  (electronic), 2005.

\bibitem[DV15]{debussche2015invariant}
Arnaud Debussche and Julien Vovelle.
\newblock Invariant measure of scalar first-order conservation laws with
  stochastic forcing.
\newblock {\em Probability Theory and Related Fields}, 163(3):575--611, 2015.

\bibitem[EKMS00]{EKhaninMazelSinai2000InvariantMeasuresBurgers}
Weinan E, Konstantin Khanin, A.~Mazel, and Yakov~G. Sinai.
\newblock Invariant measures for {{Burgers}} equation with stochastic forcing.
\newblock {\em Annals of Mathematics. Second Series}, 151(3):877--960, 2000.

\bibitem[Gao24]{gao2024polynomialwave}
Peng Gao.
\newblock Polynomial mixing for the white-forced wave equation on the whole
  line.
\newblock {\em arXiv preprint arXiv:2412.13230v2}, 2024.

\bibitem[Gao26]{Gao2026PolynomialMixingWhiteforced}
Peng Gao.
\newblock Polynomial mixing for the white-forced hyperviscous {{Burgers}}
  equation on the whole line.
\newblock {\em Journal of Differential Equations}, 480:114617, November 2026.

\bibitem[GIKP05]{GomesIturriKhaninPadill2005ViscosityLimitStationary}
Diogo Gomes, Renato Iturriaga, Konstantin Khanin, and Pablo Padilla.
\newblock Viscosity limit of stationary distributions for the random forced
  {{Burgers}} equation.
\newblock {\em Moscow Mathematical Journal}, 5(3):613--631, 2005.

\bibitem[GO14]{GrafakOh2014KatoPonceInequality}
Loukas Grafakos and Seungly Oh.
\newblock The {{Kato}}--{{Ponce}} inequality.
\newblock {\em Communications in Partial Differential Equations},
  39(6):1128--1157, June 2014.

\bibitem[Hei11]{Heil2011BasisTheoryPrimer}
Christopher Heil.
\newblock {\em A Basis Theory Primer: Expanded Edition}.
\newblock Applied and {{Numerical Harmonic Analysis}}. Birkh\"auser, Boston,
  2011.

\bibitem[HK03]{HoangKhanin2003RandomBurgersEquation}
Viet~Ha Hoang and Konstantin Khanin.
\newblock Random {{Burgers}} equation and {{Lagrangian}} systems in non-compact
  domains.
\newblock {\em Nonlinearity}, 16(3):819--842, 2003.

\bibitem[IK03]{IturriKhanin2003BurgersTurbulenceRandom}
R.~Iturriaga and Konstantin Khanin.
\newblock Burgers turbulence and random {{Lagrangian}} systems.
\newblock {\em Communications in Mathematical Physics}, 232(3):377--428, 2003.

\bibitem[IKZ20]{IturriKhaninZhang2020ExponentialConvergenceSolutions}
Renato Iturriaga, Konstantin Khanin, and Ke~Zhang.
\newblock Exponential convergence of solutions for random
  {{Hamilton}}--{{Jacobi}} equations.
\newblock {\em Stochastics and Partial Differential Equations-Analysis and
  Computations}, 8(3):544--579, September 2020.

\bibitem[KS12]{KuksinShirik2012MathematicsTwoDimensionalTurbulence}
Sergei Kuksin and Armen Shirikyan.
\newblock {\em Mathematics of Two-Dimensional Turbulence}.
\newblock Cambridge University Press, 1 edition, September 2012.

\bibitem[Kuk24]{KuksinK41TheoryTurbulence}
Sergei Kuksin.
\newblock The {{K41}} theory and turbulence in {{1D Burgers}} equation.
\newblock {\em Chaos: an Interdisciplinary Journal of Nonlinear Science},
  34(2):022103, February 2024.

\bibitem[KZ17]{KhaninZhang2017HyperbolicityMinimizersRegularity}
Konstantin Khanin and Ke~Zhang.
\newblock Hyperbolicity of minimizers and regularity of viscosity solutions for
  a random {{Hamilton}}--{{Jacobi}} equation.
\newblock {\em Communications in Mathematical Physics}, 355(2):803--837,
  October 2017.

\bibitem[Li19]{Li2019KatoPonceFractional}
Dong Li.
\newblock On kato--ponce and fractional {{Leibniz}}.
\newblock {\em Revista Matem\'atica Iberoamericana}, 35(1):23--100, January
  2019.

\bibitem[Mey92]{Meyer1992WaveletsOperators}
Yves Meyer.
\newblock {\em Wavelets and Operators}.
\newblock Number~37 in Cambridge Studies in Advanced Mathematics. Cambridge
  university press, Cambridge, transferred to digital print edition, 1992.

\bibitem[MQR21]{MatetsQuasteRemeni2021KPZFixedPoint}
Konstantin Matetski, Jeremy Quastel, and Daniel Remenik.
\newblock The {{KPZ}} fixed point.
\newblock {\em Acta Mathematica}, 227(1):115--203, 2021.

\bibitem[NZ24a]{NersesZhao2024ExponentialMixingWhiteForced}
Vahagn Nersesyan and Meng Zhao.
\newblock Exponential mixing for the white-forced complex {{Ginzburg}}--landau
  equation in the whole space.
\newblock {\em SIAM Journal on Mathematical Analysis}, 56(3):3646--3678, June
  2024.

\bibitem[NZ24b]{NersesZhao2024PolynomialMixingWhiteforced}
Vahagn Nersesyan and Meng Zhao.
\newblock Polynomial mixing for the white-forced {{Navier-Stokes}} system in
  the whole space, October 2024.

\bibitem[RS11]{RunstSickel2011SobolevSpacesFractional}
Thomas Runst and Winfried Sickel.
\newblock Sobolev spaces of fractional order, nemytskij operators, and
  nonlinear partial differential equations.
\newblock In {\em Sobolev Spaces of Fractional Order, Nemytskij Operators, and
  Nonlinear Partial Differential Equations}. De Gruyter, July 2011.

\bibitem[Shi08]{Shirik2008ExponentialMixingRandomly}
Armen Shirikyan.
\newblock Exponential mixing for randomly forced partial differential
  equations: Method of coupling.
\newblock In Claude Bardos and Andrei Fursikov, editors, {\em Instability in
  Models Connected with Fluid Flows {{II}}}, volume~7, pages 155--188. Springer
  New York, New York, NY, 2008.

\bibitem[Sin91]{Sinai1991TwoResultsConcerning}
{\relax Ya}.~G. Sina{\u \i}.
\newblock Two results concerning asymptotic behavior of solutions of the
  {{Burgers}} equation with force.
\newblock {\em Journal of Statistical Physics}, 64(1-2):1--12, 1991.

\bibitem[Tad04]{Tadmor2004BurgersEquationVanishing}
Eitan Tadmor.
\newblock Burgers' equation with vanishing hyper-viscosity.
\newblock {\em Communications in Mathematical Sciences}, 2(2):317--324, 2004.

\end{thebibliography}
\newcommand{\etalchar}[1]{$^{#1}$}

\end{document}